\pdfoutput=1
\documentclass[twoside]{article}

\usepackage{PRIMEarxiv}

\usepackage[utf8]{inputenc} 
\usepackage[T1]{fontenc}    
\usepackage{hyperref}       
\usepackage{url}            
\usepackage{booktabs}       
\usepackage{amsfonts}       
\usepackage{nicefrac}       
\usepackage{microtype}      
\usepackage{lipsum}
\usepackage{fancyhdr}       
\usepackage{graphicx}       
\graphicspath{{media/}}     

\usepackage{amsmath,amssymb,mathrsfs,amsthm}
\usepackage{multirow}
\usepackage{multicol}
\usepackage{subcaption}
\usepackage{color}
\usepackage{xcolor}
\usepackage{bm}
\usepackage{longtable}

\newtheorem{theorem}{Theorem}
\newtheorem{property}{Property}
\newtheorem{example}{Example}
\newtheorem{corollary}{Corollary}
\newtheorem{lemma}{Lemma}
\newtheorem{definition}{Definition}

\definecolor{darkgreen}{rgb}{0, 0.5, 0}

\title{Learning to Pivot as a Smart Expert
}

\author{
  Tianhao Liu \\
  Research Institute for Interdisciplinary Sciences \\
  Shanghai University of Finance and Economics \\
  \texttt{liu.tianhao@163.sufe.edu.cn} \\
   \And
  Shanwen Pu \\
  Research Institute for Interdisciplinary Sciences \\
  Shanghai University of Finance and Economics \\
  \texttt{2019212802@live.sufe.edu.cn} \\
   \And
  Dongdong Ge \\
  Research Institute for Interdisciplinary Sciences \\
  Shanghai University of Finance and Economics \\
  \texttt{ge.dongdong@mail.shufe.edu.cn} \\
   \And
  Yinyu Ye \\
  Stanford University \\
  \texttt{yyye@stanford.edu} \\
}

\begin{document}
\global\long\def\inprod#1#2{\left\langle #1,#2\right\rangle }%
\global\long\def\inner#1#2{\langle#1,#2\rangle}%
\global\long\def\binner#1#2{\big\langle#1,#2\big\rangle}%
\global\long\def\norm#1{\Vert#1\Vert}%
\global\long\def\bnorm#1{\big\Vert#1\big\Vert}%
\global\long\def\Bnorm#1{\Big\Vert#1\Big\Vert}%
\global\long\def\red#1{\textcolor{red}{#1}}%
\global\long\def\blue#1{\textcolor{blue}{#1}}%

\global\long\def\brbra#1{\big(#1\big)}%
\global\long\def\Brbra#1{\Big(#1\Big)}%
\global\long\def\rbra#1{(#1)}%

\global\long\def\sbra#1{[#1]}%
\global\long\def\bsbra#1{\big[#1\big]}%
\global\long\def\Bsbra#1{\Big[#1\Big]}%
\global\long\def\abs#1{\vert#1\vert}%
\global\long\def\babs#1{\big\vert#1\big\vert}%

\global\long\def\cbra#1{\{#1\}}%
\global\long\def\bcbra#1{\big\{#1\big\}}%
\global\long\def\Bcbra#1{\Big\{#1\Big\}}%
\global\long\def\vertiii#1{\left\vert \kern-0.25ex  \left\vert \kern-0.25ex  \left\vert #1\right\vert \kern-0.25ex  \right\vert \kern-0.25ex  \right\vert }%
\global\long\def\matr#1{\bm{#1}}%
\global\long\def\til#1{\tilde{#1}}%
\global\long\def\wtil#1{\widetilde{#1}}%
\global\long\def\wh#1{\widehat{#1}}%
\global\long\def\mcal#1{\mathcal{#1}}%
\global\long\def\mbb#1{\mathbb{#1}}%
\global\long\def\mtt#1{\mathtt{#1}}%
\global\long\def\ttt#1{\texttt{#1}}%
\global\long\def\dtxt{\textrm{d}}%
\global\long\def\bignorm#1{\bigl\Vert#1\bigr\Vert}%
\global\long\def\Bignorm#1{\Bigl\Vert#1\Bigr\Vert}%
\global\long\def\rmn#1#2{\mathbb{R}^{#1\times#2}}%
\global\long\def\deri#1#2{\frac{d#1}{d#2}}%
\global\long\def\pderi#1#2{\frac{\partial#1}{\partial#2}}%
\global\long\def\limk{\lim_{k\rightarrow\infty}}%
\global\long\def\trans{\textrm{T}}%
\global\long\def\onebf{\mathbf{1}}%
\global\long\def\Bbb{\mathbb{B}}%
\global\long\def\hbf{\mathbf{h}}%
\global\long\def\zerobf{\mathbf{0}}%
\global\long\def\zero{\bm{0}}%

\global\long\def\Euc{\mathrm{E}}%
\global\long\def\Expe{\mathbb{E}}%
\global\long\def\rank{\mathrm{rank}}%
\global\long\def\range{\mathrm{range}}%
\global\long\def\diam{\mathrm{diam}}%
\global\long\def\epi{\mathrm{epi} }%
\global\long\def\inte{\operatornamewithlimits{int}}%
\global\long\def\dist{\operatornamewithlimits{dist}}%
\global\long\def\proj{\operatorname{Proj}}%
\global\long\def\cov{\mathrm{Cov}}%
\global\long\def\argmin{\operatornamewithlimits{argmin}}%
\global\long\def\argmax{\operatornamewithlimits{argmax}}%
\global\long\def\tr{\operatornamewithlimits{tr}}%
\global\long\def\dis{\operatornamewithlimits{dist}}%
\global\long\def\sign{\operatornamewithlimits{sign}}%
\global\long\def\prob{\mathrm{Prob}}%
\global\long\def\st{\operatornamewithlimits{s.t.}}%
\global\long\def\dom{\mathrm{dom}}%
\global\long\def\prox{\mathrm{prox}}%
\global\long\def\diag{\mathrm{diag}}%
\global\long\def\and{\mathrm{and}}%
\global\long\def\aleq{\overset{(a)}{\leq}}%
\global\long\def\aeq{\overset{(a)}{=}}%
\global\long\def\ageq{\overset{(a)}{\geq}}%
\global\long\def\bleq{\overset{(b)}{\leq}}%
\global\long\def\beq{\overset{(b)}{=}}%
\global\long\def\bgeq{\overset{(b)}{\geq}}%
\global\long\def\cleq{\overset{(c)}{\leq}}%
\global\long\def\ceq{\overset{(c)}{=}}%
\global\long\def\cgeq{\overset{(c)}{\geq}}%
\global\long\def\dleq{\overset{(d)}{\leq}}%
\global\long\def\deq{\overset{(d)}{=}}%
\global\long\def\dgeq{\overset{(d)}{\geq}}%
\global\long\def\eleq{\overset{(e)}{\leq}}%
\global\long\def\eeq{\overset{(e)}{=}}%
\global\long\def\egeq{\overset{(e)}{\geq}}%
\global\long\def\fleq{\overset{(f)}{\leq}}%
\global\long\def\feq{\overset{(f)}{=}}%
\global\long\def\fgeq{\overset{(f)}{\geq}}%
\global\long\def\gleq{\overset{(g)}{\leq}}%
\global\long\def\as{\textup{a.s.}}%
\global\long\def\ae{\textup{a.e.}}%
\global\long\def\Var{\operatornamewithlimits{Var}}%
\global\long\def\clip{\operatorname{clip}}%
\global\long\def\conv{\operatorname{conv}}%
\global\long\def\Cov{\operatornamewithlimits{Cov}}%
\global\long\def\raw{\rightarrow}%
\global\long\def\law{\leftarrow}%
\global\long\def\Raw{\Rightarrow}%
\global\long\def\Law{\Leftarrow}%
\global\long\def\vep{\varepsilon}%
\global\long\def\dom{\operatornamewithlimits{dom}}%
\global\long\def\tsum{{\textstyle {\sum}}}%
\global\long\def\Cbb{\mathbb{C}}%
\global\long\def\Ebb{\mathbb{E}}%
\global\long\def\Fbb{\mathbb{F}}%
\global\long\def\Nbb{\mathbb{N}}%
\global\long\def\Rbb{\mathbb{R}}%
\global\long\def\extR{\widebar{\mathbb{R}}}%
\global\long\def\Pbb{\mathbb{P}}%
\global\long\def\Zbb{\mathbb{Z}}%
\global\long\def\Mrm{\mathrm{M}}%
\global\long\def\Acal{\mathcal{A}}%
\global\long\def\Bcal{\mathcal{B}}%
\global\long\def\Ccal{\mathcal{C}}%
\global\long\def\Dcal{\mathcal{D}}%
\global\long\def\Ecal{\mathcal{E}}%
\global\long\def\Fcal{\mathcal{F}}%
\global\long\def\Gcal{\mathcal{G}}%
\global\long\def\Hcal{\mathcal{H}}%
\global\long\def\Ical{\mathcal{I}}%
\global\long\def\Kcal{\mathcal{K}}%
\global\long\def\Lcal{\mathcal{L}}%
\global\long\def\Mcal{\mathcal{M}}%
\global\long\def\Ncal{\mathcal{N}}%
\global\long\def\Ocal{\mathcal{O}}%
\global\long\def\Pcal{\mathcal{P}}%
\global\long\def\Scal{\mathcal{S}}%
\global\long\def\Tcal{\mathcal{T}}%
\global\long\def\Xcal{\mathcal{X}}%
\global\long\def\Ycal{\mathcal{Y}}%
\global\long\def\Zcal{\mathcal{Z}}%
\global\long\def\i{i}%

\global\long\def\abf{\mathbf{a}}%
\global\long\def\bbf{\mathbf{b}}%
\global\long\def\cbf{\mathbf{c}}%
\global\long\def\fbf{\mathbf{f}}%
\global\long\def\qbf{\mathbf{q}}%
\global\long\def\gbf{\mathbf{g}}%
\global\long\def\ebf{\mathbf{e}}%
\global\long\def\lambf{\bm{\lambda}}%
\global\long\def\alphabf{\bm{\alpha}}%
\global\long\def\sigmabf{\bm{\sigma}}%
\global\long\def\thetabf{\bm{\theta}}%
\global\long\def\deltabf{\bm{\delta}}%
\global\long\def\lbf{\mathbf{l}}%
\global\long\def\ubf{\mathbf{u}}%
\global\long\def\pbf{\mathbf{\mathbf{p}}}%
\global\long\def\vbf{\mathbf{v}}%
\global\long\def\wbf{\mathbf{w}}%
\global\long\def\xbf{\mathbf{x}}%
\global\long\def\ybf{\mathbf{y}}%
\global\long\def\zbf{\mathbf{z}}%
\global\long\def\dbf{\mathbf{d}}%
\global\long\def\Wbf{\mathbf{W}}%
\global\long\def\Abf{\mathbf{A}}%
\global\long\def\Ubf{\mathbf{U}}%
\global\long\def\Pbf{\mathbf{P}}%
\global\long\def\Ibf{\mathbf{I}}%
\global\long\def\Ebf{\mathbf{E}}%
\global\long\def\sbf{\mathbf{s}}%
\global\long\def\Mbf{\mathbf{M}}%
\global\long\def\Nbf{\mathbf{N}}%
\global\long\def\Dbf{\mathbf{D}}%
\global\long\def\Qbf{\mathbf{Q}}%
\global\long\def\Lbf{\mathbf{L}}%
\global\long\def\Pbf{\mathbf{P}}%
\global\long\def\Xbf{\mathbf{X}}%
\global\long\def\Bbf{\mathbf{B}}%
\global\long\def\zerobf{\mathbf{0}}%
\global\long\def\onebf{\mathbf{1}}%


\global\long\def\abm{\bm{a}}%
\global\long\def\bbm{\bm{b}}%
\global\long\def\cbm{\bm{c}}%
\global\long\def\dbm{\bm{d}}%
\global\long\def\ebm{\bm{e}}%
\global\long\def\fbm{\bm{f}}%
\global\long\def\gbm{\bm{g}}%
\global\long\def\hbm{\bm{h}}%
\global\long\def\pbm{\bm{p}}%
\global\long\def\qbm{\bm{q}}%
\global\long\def\rbm{\bm{r}}%
\global\long\def\sbm{\bm{s}}%
\global\long\def\tbm{\bm{t}}%
\global\long\def\ubm{\bm{u}}%
\global\long\def\vbm{\bm{v}}%
\global\long\def\wbm{\bm{w}}%
\global\long\def\xbm{\bm{x}}%
\global\long\def\ybm{\bm{y}}%
\global\long\def\zbm{\bm{z}}%
\global\long\def\Abm{\bm{A}}%
\global\long\def\Bbm{\bm{B}}%
\global\long\def\Cbm{\bm{C}}%
\global\long\def\Dbm{\bm{D}}%
\global\long\def\Ebm{\bm{E}}%
\global\long\def\Fbm{\bm{F}}%
\global\long\def\Gbm{\bm{G}}%
\global\long\def\Hbm{\bm{H}}%
\global\long\def\Ibm{\bm{I}}%
\global\long\def\Jbm{\bm{J}}%
\global\long\def\Lbm{\bm{L}}%
\global\long\def\Obm{\bm{O}}%
\global\long\def\Pbm{\bm{P}}%
\global\long\def\Qbm{\bm{Q}}%
\global\long\def\Rbm{\bm{R}}%
\global\long\def\Ubm{\bm{U}}%
\global\long\def\Vbm{\bm{V}}%
\global\long\def\Wbm{\bm{W}}%
\global\long\def\Xbm{\bm{X}}%
\global\long\def\Ybm{\bm{Y}}%
\global\long\def\Zbm{\bm{Z}}%
\global\long\def\lambm{\bm{\lambda}}%
\global\long\def\alphabm{\bm{\alpha}}%
\global\long\def\albm{\bm{\alpha}}%
\global\long\def\taubm{\bm{\tau}}%
\global\long\def\mubm{\bm{\mu}}%
\global\long\def\inftybm{\bm{\infty}}%
\global\long\def\yrm{\mathrm{y}}%

\maketitle

\begin{abstract}
  Linear programming has been practically solved mainly by simplex and interior point methods. Compared with the weakly polynomial complexity obtained by the interior point methods, the existence of strongly polynomial bounds for the length of the pivot path generated by the simplex methods remains a mystery. In this paper, we propose two novel pivot experts that leverage both global and local information of the linear programming instances for the primal simplex method and show their excellent performance numerically. The experts can be regarded as a benchmark to evaluate the performance of classical pivot rules, although they are hard to directly implement. To tackle this challenge, we employ a graph convolutional neural network model, trained via imitation learning, to mimic the behavior of the pivot expert. Our pivot rule, learned empirically, displays a significant advantage over conventional methods in various linear programming problems, as demonstrated through a series of rigorous experiments.

\end{abstract}


\section{Introduction}\label{sec:intro}

Linear programming (LP) is among the most fundamental problems that has been well-studied in the field of optimization. LP is not only directly used across various industries but has also become an important cornerstone of mixed integer programming (MIP) and sequential linear programming (SLP) methods for solving nonlinear programming (NLP). Nowadays, most commercial \cite{copt, gurobi, nickel2022ibm, xpress2014fico} and open-source \cite{huangfu2018parallelizing, achterberg2009scip} solvers have implemented fast and stable LP solvers (software for solving LP) and constantly achieve new advancements for large-scale LP problems.

A general LP formulation solves the problem with only a linear objective function and linear constraints. It is well-known that these linear constraints geometrically form a polyhedron and if the optimal solution exists, it exists in one of the vertices which belong to basic solutions from an algebra point of view. The state-of-the-art methods for LP include the simplex methods, the interior-point methods (IPMs), and some recently developed first-order methods (FOMs) \cite{applegate2021practical, deng2022new}. For high accuracy and reliability, the simplex methods and IPMs are preferred and become two main classes of algorithms implemented in commercial solvers.

The simplex methods start from a basic solution (notice that some simplex methods do not require either primal or dual feasibility) and improve objective or feasibility by reaching a certain adjacent basic solution, which is called the pivot. The criterion for switching from the current basic solution to its neighbor is called the pivot rule. Different pivot rules greatly affect the performance of the simplex methods, so designing a smart pivot rule is one of the most significant simplex's tasks. From the theoretical aspect, whether there exists a strongly polynomial bound for the length of the pivot path, which represents the iteration number of the simplex methods, attracts much research interest but is still an open problem. Although simplex methods include different types (i.e., the primal, the dual, and the primal-dual simplex methods), most modern commercial solvers pay more attention to the primal and dual simplex methods.

Instead of moving between vertices, IPMs keep an interior point and walk along a central path approaching the optimal solution. Weakly polynomial complexity and excellent practical performance for LP are first achieved by IPMs simultaneously \cite{karmarkar1984new, ye2011interior, mehrotra1992implementation}. In practice, IPMs usually yield a dense primal-dual approximate solution. Modern commercial solvers tend to conduct a crossover from IPMs solution to a basic solution and run simplex methods for a sparse exact solution.

In this paper, we focus on designing smart pivot rules for the primal simplex method. It is believed that our study can be easily migrated to other types of simplex methods. The smart pivot rules are expected to generate short pivot paths for different kinds of LP instances in different scales and should not run intolerably slow.

\paragraph{Our contribution}
We propose a class of novel pivot experts that can be observed to outperform several classical and popular pivot rules. To modify the experts for practical use, we also apply machine learning methods to imitate the pivot expert. Our contribution can be summarized as follows.

\begin{itemize}
    \item First, we design novel pivot experts. Compared with classical pivot rules that only utilize local information, we consider a smart pivot rule should be able to combine global and local information together. Based on this idea, two pivot experts are proposed and can generate significantly shorter pivot paths than classical pivot candidates in a series of experiments. 
    The pivot paths generated by experts are also analyzed on Klee-Minty variants.

    \item Second, to the best of our knowledge, this paper is the first to combine imitation learning with dynamic pivot for general LP of different scales. Incorporating a graph convolutional neural network (GCNN) model, our learning aims to predict the experts' pivot behavior, which removes the requirement for global information. Experiments show that our practical version of pivot expert overall generates shorter pivot paths.
\end{itemize}


\paragraph{Organization of the paper}
The paper is organized as follows. Section \ref{sec:review} reviews related studies in simplex methods and machine learning methods to help optimization, especially for LP. Section \ref{sec:expert} describes the novel class of pivot experts and discusses its merits and demerits. Section \ref{sec:learn} provides an imitation learning method to help our idea of experts become practical. Section \ref{sec:experiments} presents twofold experiments to verify the superiority of our pivot experts and the learned pivot rule. Section \ref{sec:conclusion} concludes the paper and discusses some relative topics.

\paragraph{Notations}
Several commonly used notations are listed below. We use bold letters for vectors and matrices. Let $\Rbb^n$ denote the $n$-dimensional Euclidean space. We use $\bar{\Rbb}$ and $\underline{\Rbb}$ to express $\Rbb\cup\{+\infty\}$ and $\Rbb\cup\{-\infty\}$. Let $\xbf_j$ be the $j$th element of vector $\xbf$. We use $\xbf \geq \ybf$ to express the element-wise inequality $\xbf_{i} \geq \ybf_{i}$. Let $\zerobf$, $\onebf$, and $\inftybm$ be a vector of zeros, ones, and infinities. Let $\Ibf$ be the identity matrix and $\ebf_j$ be the $j$th column of $\Ibf$. The dimension of a vector or a matrix will be unspecified whenever it is clear from the context. $\|\cdot\|_\ell$ is $\ell$-norm (2-norm if $\ell$ is omitted) while $\abs{\cdot}$ is absolute value. Let $\Abf_{i,j}$ be the entry in the $i$th row and $j$th column of matrix $\Abf$. Let $\Abf_{j}$ be the $j$th column of matrix $\Abf$, and $\Abf_{\mathcal{I}}$ be the matrix formulated by columns $\Abf_{j}$ for $j\in\mathcal{I}$. Let $\Abf_{i,:}$ be the $i$th row of matrix $\Abf$.

\section{Related Work}\label{sec:review}

\subsection{Simplex Methods}

In this subsection, we will describe a series of pivot rules for LP in standard form
\begin{equation}\label{prob:std-lp}
    \begin{aligned}
        \min\ & \cbf^\top \xbf \\
        \st\ & \Abf\xbf = \bbf \\
        & \xbf \geq \zerobf,
    \end{aligned}
\end{equation}
where $\xbf\in\Rbb^n$ and $\Abf \in \Rbb^{m\times n}, \bbf \in \Rbb^m, \cbf\in\Rbb^n$. Let $\Bbf$ and $\Nbf$ be indices of basic and non-basic variables. This formulation is used for academic research, but will not be preferred in modern LP solvers. Our implementation of pivot experts considers a more practical formulation which will be described in Section \ref{sec:expert} later.

\paragraph{Simplex methods}
Studies on LP and the simplex methods date back to the very beginning of modern operation research. In 1947, Dantzig \cite{dantzig1963linear} revealed the significance of LP and provided the framework of the primal simplex method. The primal simplex method uses Phase I to achieve a basic primal feasible solution and maintains primal feasibility during Phase II to solve the LP instance. The dual simplex method solves the LP in the view of duality, which gets a basic dual feasible solution and improves primal feasibility during Phase II. The primal-dual simplex method aims to get rid of Phase I and improves primal or dual feasibility at each iteration. For more detail, we highly recommend Pan's book \cite{pan_linear_2023} for basic knowledge and recent development of LP. In this paper, we concentrate on the primal simplex method.

\paragraph{Pivot rules in primal simplex method}
The simplex methods switch between adjacent basic solutions, which is called the pivot, at each iteration. Pivot is algebraically selecting a non-basic variable to enter the basis, then conducting a ratio test evaluating distance to go, and letting one basic variable leave the basis. How to choose among adjacent basic solutions (i.e., to design a pivot rule) is key to a successful simplex method.

When Dantzig proposes primal simplex method, he also provides a pivot rule to choose candidate with the most negative reduced cost $\bar{\cbf}_{j} = \cbf_{j} - \cbf_{\Bbf}^\top\Abf_{\Bbf}^{-1}\Abf_{j}$, which is called Dantzig's rule. To avoid cycling in a pivot path, the famous Bland's rule \cite{bland1977new} of choosing the candidate with minimum index and other lexicographic pivot rules are proposed. These anti-cycling rules make simplex terminate in finite steps but perform poorly in real applications. Since other practically anti-cycling methods like perturbation work well, more practical interest is attracted by generating shorter pivot paths rather than theoretically finite termination.

The most widely used pivot rule in modern simplex solvers is the steepest-edge rule \cite{goldfarb1977practicable}. It has been observed to generate relatively short pivot paths. The steepest-edge rule chooses candidate with the most negative $\frac{\bar{\cbf}_{j}}{\sqrt{\|\Abf_{\Bbf}^{-1}\Abf_{j}\|^2+1}}$. If non-basic variable $\xbf_{j}$ enters the basis, the point will move along $\dbf = \begin{bmatrix}
    \dbf_{\Bbf} \\ \dbf_{\Nbf}
\end{bmatrix} = \begin{bmatrix}
    -\Abf_{\Bbf}^{-1}\Abf_{j} \\ \ebf_{j}
\end{bmatrix}$. Calculating $\cos\langle \cbf, \dbf \rangle = \frac{1}{\|\cbf\|}\frac{\bar{\cbf}_j}{\sqrt{\|\Abf_{\Bbf}^{-1}\Abf_{j}\|^2+1}}$ and noticing that $\|\cbf\|$ is a constant, steepest-edge can be explained as always moving in the descending direction most parallel to $\cbf$. Further study for the steepest-edge rule \cite{forrest1992steepest} includes efficient implementation, choice of the norm, dual version of steepest-edge, etc.

Another pivot rule called the greatest improvement rule \cite{jeroslow1973simplex} can also generate short pivot paths. Like strong branching \cite{achterberg2005branching} in MIP, the greatest improvement rule prefers a candidate that brings the greatest improvement in objective value to enter the basis. However, sometimes too much greed is not the best option (as will be seen in Section \ref{sec:experiments}). Besides, to calculate the improvement, a ratio test for each candidate variable is needed, which is often too expensive for the simplex method.

As the increasing scale of LP has hit the limits of computer power, the rules mentioned above become inefficient because of either slow calculation or intolerably long paths. Devex rule \cite{harris1973pivot} and the largest distance rule \cite{roos1986pivoting, pan2008largest} are therefore proposed. Devex rule inexactly approximates the score in steepest-edge and thus can update faster. The largest distance rule regards negative reduced cost $\bar{\cbf}_j$ as violating the $j$th constraint $\Abf_{j}^\top \ybf \leq \cbf_{j}$ in the dual problem of LP. Therefore, we can use the distance formula from point to hyperplane as a degree of violation to calculate the largest distance score $\frac{\bar{\cbf}}{\|\Abf_{j}\|}$ and choose the most negative one to enter the basis. Notice that the denominator stays the same and only requires to be calculated once.

\paragraph{Worst cases of pivot rules}
With so many rules accumulated and a wide variety of manifestations observed in practice, researchers are puzzled by the complexity of the simplex methods. In other words, what is the worst possible path for the simplex methods? Or generally, can LP be solved with strongly polynomial algorithms?

These problems are surprisingly difficult to give a general satisfying answer even though we have already known that LP has weakly polynomial bounds guaranteed by IPMs. In practice, the pivot numbers are observed to be polynomial with respect to the number of variables and constraints. Some researchers believe that simplex can lead to a strongly polynomial bound for LP. Unfortunately, worst cases with pivot paths of exponential lengths have been discovered for most deterministic pivot rules \cite{klee1972good, avis1978notes, goldfarb1979worst, roos1990exponential}. After introducing randomization and parameterized LP, several sub-exponential bounds \cite{matouvsek1992subexponential, kalai1992subexponential} or weakly polynomial bounds \cite{kelner2006randomized} have been derived for general LP. For some special LP classes, certain pivot rules are proved to be strongly polynomial \cite{ye2011simplex, kitahara2013bound}. In short, analyzing the complexity of simplex is still a long way to go.

\subsection{Machine Learning for Mathematical Optimization}

The development of optimization helps smarter machine learning models train better for larger and harder tasks. Interestingly, in the opposite direction, machine learning methods also make efforts to accelerate or improve optimization methods. This field of research can be called machine learning for mathematical optimization (ML4MO). In this subsection, we review several ways to apply machine learning to optimization. Machine learning models in ML4MO can be classified according to the encoding methods of the optimization problem as classical models and GCNN models. Learning methods used in ML4MO lie in supervised learning, imitation learning, and reinforcement learning. Many studies in ML4MO help solve MIP since it is a harder and more common problem. Although there have not been many fusions of machine learning and LP, some preliminary attempts are made and deserve mention.

\paragraph{Classical model-based ML4MO}
Early studies extract feature vectors from problems and use classical models for prediction. Di Liberto et al. \cite{di2016dash} extract feature vectors of nodes in branch-and-bound (BnB) trees for the clustering method to adaptively select branching rules during BnB in MIP solving procedure. Alvarez et al. \cite{alvarez2017machine} use ExtraTrees to imitate strong branching behavior for MIP. He et al. \cite{he2014learning} learn a node selection strategy from the shortest path in BnB trees built by SCIP towards the node containing the optimal solution. They conduct a DAgger framework for better learning results. Khalil et al. \cite{khalil2016learning} imitate strong branching via learning a variable ranking function by support vector machines (SVM) rather than directly scoring all candidates. Zarpellon et al. \cite{zarpellon2021parameterizing} add more features of BnB trees themselves for learning a more general branching rule. Berthold et al. \cite{berthold2022learning} use random forest to decide whether to use local cuts when solving MIP.

\paragraph{GCNN model-based ML4MO}
A new way to encode problems is proposed by Gasse et al. \cite{gasse2019exact} to imitate strong branching. They creatively encode MIP to be a bipartite graph. The bipartite graph contains almost all information of the original problem and thus avoids loss of information which the classical methods often suffer from. The graph has two groups of nodes representing variables and constraints. An edge with feature $\Abf_{i,j}$ links a variable and a constraint if the coefficient $\Abf_{i,j}$ of the variable in the constraint is non-zero. Other information such as objective and right-hand side of constraints will be added to the features of nodes. With a graph as input, GCNN models are naturally adopted to perform more comprehensive feature extraction and are ready for subsequent models to make final decisions.

Gasse's work inspires a new stream of ML4MO. Gupta et al. \cite{gupta2020hybrid} replace GCNN in the BnB child nodes with a more CPU-friendly multi-layer perceptron (MLP) model so that the strategy can run more efficiently on large-scale problems. Ding et al. \cite{ding2020accelerating} modify the bipartite graph by regarding the objective function as a single node to form a tripartite graph to predict partial solutions to generate branching hyperplanes. Nair et al. \cite{nair2020solving} use GCNN to design two strategies including neural diving and neural branching. Neural diving predicts a partial solution and gets a feasible solution from sub-MIP for the warm start of BnB. Neural branching appends SelectiveNet after GCNN to imitate strong branching. Sonnerat et al. \cite{sonnerat2021learning} refine Nair's neural diving by adding a large neighborhood search (LNS) for a better solution. Paulus et al. \cite{paulus2023learning} predict a solution by GCNN and design a diving method according to the predicted solution.

\paragraph{Machine learning for LP and simplex}
Several attempts have been made to accelerate LP utilizing machine learning methods. Most of them study pivot in the primal simplex method. Adham et al. \cite{adham2021machine} use boosted trees and neural networks to predict the best pivot rule for each LP instance but the approach is a one-shot decision and lacks flexibility. Suriyanarayana et al. \cite{suriyanarayana2022reinforcement} use reinforcement learning to dynamically switch between Dantzig's rule and steepest-edge rule for solving LP relaxation of non-Euclidean TSPs with five cities. However, it is only proof of concept that is not suitable for larger problems or problems with different scales. Li et al. \cite{li2022rethinking} use Monte Carlo tree search (MCTS) to directly decide which candidate will enter the basis. For each new LP instance, MCTS explores slowly at every single pivot.

The state in both Suriyanarayana's and Li's reinforcement learning approaches is based on simplex tableau directly, which is not scalable for large-scale LP. The bipartite graph and GCNN can be a more reasonable tool to encode LP for its permutation invariance and scalability. In theory, Chen et al. \cite{chen2022representing} reveal the potential power of GCNN in distinguishing LP with different characteristics. In practice, Fan et al. \cite{fan2023smart} use GCNN to predict a better initial basis for the primal simplex method, which is the preparatory work for Phase II.

\paragraph{Machine learning for column generation}
Column generation (CG) can be viewed as a generalized version of the simplex. The most similar learning framework to ours is in this category. Morabit et al. \cite{morabit2021machine} use GCNN to select better columns in CG oracle while ours use GCNN to decide a candidate to enter the basis. However, there exist some differences between Morabit's work and ours. First, CG deals with MIP and is more like a partial pricing version of the simplex method. Our work focuses on improving the full pricing pivot rule. Second, the expert rule in Morabit's work prefers improvement in objective value with fewer columns at each CG iteration, which is more like the greatest improvement rule. In our experiments, the greatest improvement rule is not an expert in the simplex method while our pivot experts which leverage both global and local information perform significantly better.

\section{Smart Pivot Experts}\label{sec:expert}

\subsection{Primal Simplex Method}
We consider a general LP formulation
\begin{equation} \label{prob:lp}
    \begin{aligned}
        \min\ & \cbf^\top \xbf \\
        \st\ & \Abf\xbf = \bbf \\
        & \lbf \leq \xbf \leq \ubf,
    \end{aligned}
\end{equation}
where $\xbf\in\Rbb^n$ and $\Abf \in \Rbb^{m\times n}, \bbf \in \Rbb^m, \cbf\in\Rbb^n,\lbf\in\underline{\Rbb}^n,\ubf\in\bar{\Rbb}^n$. \eqref{prob:lp} is more user-friendly than \eqref{prob:std-lp}, so we will derive our pivot experts and conduct experiments based on this formulation.

\paragraph{Phase I: find a basic feasible solution}
Phase I aims to find a basic feasible solution for Phase II in the primal simplex method. In fact, multiple ways including the big-M method and heuristics are designed to achieve the purpose. However, those methods form an independent topic that is beyond our discussion. Here we consider the classical auxiliary LP method.

Let $\xbf_j^0 = \begin{cases}
    \ubf_j, & \text{if } \lbf_j = -\infty \text{ and } \ubf_j < +\infty \\
    \lbf_j, & \text{if } \lbf_j > -\infty \\
    0, & \text{otherwise}
\end{cases}$, and add artificial variables $\zbf$ to form the auxiliary LP \eqref{prob:alp}
\begin{equation}
    \begin{aligned}
        \min\ & \onebf^\top \zbf \\
        \st\ & \Abf\xbf + \Ebf \zbf = \bbf \\
        & \lbf \leq \xbf \leq \ubf \\
        & \zbf \geq \zerobf,
    \end{aligned} \label{prob:alp}
\end{equation}
where $\Ebf$ is a diagonal matrix with $\Ebf_{i,i} = \begin{cases}
    1, & \text{if } \bbf_i - \Abf_{i,:}\xbf^0 \geq 0 \\
    -1, & \text{otherwise}
\end{cases}$. Let $\zbf^0_i = \abs{\bbf_i - \Abf_{i,:}\xbf^0}$, we have a basic feasible solution $(\xbf^0,\zbf^0)$ for \eqref{prob:alp}. Then we apply the primal simplex method and check the optimal value. If the optimal value is 0, we gradually pivot to let the remaining basic $\zbf$ leave the basis and remove redundant constraints to get a basic feasible solution for \eqref{prob:lp}. Otherwise, the original problem is primal infeasible.

\paragraph{Phase II: solve the LP}
Phase II starts to solve \eqref{prob:lp} after a basic feasible solution is obtained from Phase I. At each pivot, we check the reduced cost $\bar{\cbf}$ of those non-basic variables and pick those candidates which have negative (positive) $\bar{\cbf}$ and are equal to its lower (upper) bound.

A variable to enter basis is selected according to a certain pivot rule and afterwards a ratio test is conducted. If there occurs a tie, we will choose the one with minimum (maximum) index to enter (leave) the basis. Then the inverse of the basic matrix and reduced cost are updated for the next pivot. The method terminates when dual feasibility is reached (i.e., no variable can enter the basis).

Currently, there is no satisfying simplex solver that supports dynamically switching among various pivot rules at each pivot. Therefore, we implement the above two-phase revised simplex in our own primal simplex solver prototype. More details are presented in Appendix \ref{apx:solver}.

\subsection{Designing Smart Pivot Experts}
The existing pivot rules all consider only local information. Here the term "local" refers to the information that describes the landscape around the current basic feasible solution. If a pivot rule makes a myopic decision, the basic feasible solution may lead to being stuck in a rugged area in the future.

The main idea for designing a smart pivot expert is to provide global information for it. Some trials proposed by others include tree search for future information \cite{li2022rethinking}, using interior point information \cite{todd1990dantzig, roos1986pivoting, tamura1988dual}, or choosing more than one variable at one time to enter the basis \cite{yang2020double}. However, there is something the most global but easily overlooked ---- the optimal basis.

It is natural to doubt that there is no need to solve the LP if an optimal basis is known. Our point is that we regard solved LP instances as historical data and want to obtain smart pivot experts from them. The absence of optimal basis only prevents us from directly using these experts, but that is a separate issue from "how expert a pivot rule will be if it knows global information". From the data-driven perspective, the challenge can be overcomed in Section \ref{sec:learn} by imitation learning. In this subsection, we purely focus on designing smart pivot experts.

With the optimal basis at each iteration, the smart pivot rule can be guided by the following two goals:
\begin{enumerate}
    \item When selecting a candidate to enter the basis, a smart pivot rule should let the basic variable in the optimal basis enter first.
    \item To select a variable to leave the basis when a tie occurs in the ratio test, a smart pivot rule should let the non-basic in the optimal basis leave first.
\end{enumerate}
The smart pivot rule will greedily bring the current basis as close to the optimal basis (in terms of the difference in the basic indices) as possible from a global perspective. Theorem \ref{thm:idx-in} guarantees that such a variable can always be found as long as the objective value is not optimal.
\begin{theorem}\label{thm:idx-in}
    Given the optimal basis, if the current objective value is not optimal, there must exist a variable in the optimal basis that can enter the current basis immediately.
\end{theorem}
\begin{proof}
Let $(B^k, N_\lbf^k, N_\ubf^k)$ be the partitions of basic, non-basic reaching lower bound and non-basic reaching upper bound variables at iteration $k$ and $(B^*, N_\lbf^*, N_\ubf^*)$ be the partitions of basic and non-basic variables of the optimization solution $\xbf^*$. Let $\xbf^*$ be partitioned using $(B^k, N_\lbf^k, N_\ubf^k)$, i.e.,
\begin{equation*}
    \xbf^* = \begin{bmatrix}
        \xbf^*_{B^k} \\ \xbf^*_{N^k_\lbf} \\ \xbf^*_{N^k_\ubf}
    \end{bmatrix}.
\end{equation*}
We have
\begin{equation*}
    \xbf^*_{B^k} = \Abf_{B^k}^{-1}\bbf - \Abf_{B^k}^{-1}\Abf_{N^k_\lbf}\xbf^*_{N^k_\lbf} - \Abf_{B^k}^{-1}\Abf_{N^k_\ubf}\xbf^*_{N^k_\ubf}.
\end{equation*}

Calculate objective function
\begin{align*}
    \cbf^\top\xbf^* &= \cbf^\top_{B^k}\xbf^*_{B^k} + \cbf^\top_{N^k_\lbf}\xbf^*_{N^k_\lbf} + \cbf^\top_{N^k_\ubf}\xbf^*_{N^k_\ubf} \\
    &= \cbf^\top_{B^k}\Abf_{B^k}^{-1}\bbf - \cbf^\top_{B^k}\Abf_{B^k}^{-1}\Abf_{N^k_\lbf}\xbf^*_{N^k_\lbf} - \cbf^\top_{B^k}\Abf_{B^k}^{-1}\Abf_{N^k_\ubf}\xbf^*_{N^k_\ubf} + \cbf^\top_{N^k_\lbf}\xbf^*_{N^k_\lbf} + \cbf^\top_{N^k_\ubf}\xbf^*_{N^k_\ubf} \\
    &= \cbf^\top_{B^k}\left(\Abf_{B^k}^{-1}\bbf - \Abf_{B^k}^{-1}\Abf_{N^k_\lbf}\lbf_{N^k_\lbf} - \Abf_{B^k}^{-1}\Abf_{N^k_\ubf}\ubf_{N^k_\ubf}\right) + \cbf^\top_{N^k_\lbf}\lbf_{N^k_\lbf} + \cbf^\top_{N^k_\ubf}\ubf_{N^k_\ubf} \\
    &\quad - \cbf^\top_{B^k}\Abf_{B^k}^{-1}\Abf_{N^k_\lbf}\left(\xbf^*_{N^k_\lbf} - \lbf_{N^k_\lbf} \right) - \cbf^\top_{B^k}\Abf_{B^k}^{-1}\Abf_{N^k_\ubf}\left(\xbf^*_{N^k_\ubf} - \ubf_{N^k_\ubf} \right) \\
    &\quad + \cbf^\top_{N^k_\lbf}\left(\xbf^*_{N^k_\lbf} - \lbf_{N^k_\lbf} \right) + \cbf^\top_{N^k_\ubf}\left(\xbf^*_{N^k_\ubf} - \ubf_{N^k_\ubf} \right) \\
    &= \cbf^\top\xbf^k + \bar{\cbf}^\top_{N^k_\lbf}\left(\xbf^*_{N^k_\lbf} - \lbf_{N^k_\lbf} \right) + \bar{\cbf}^\top_{N^k_\ubf}\left(\xbf^*_{N^k_\ubf} - \ubf_{N^k_\ubf} \right),
\end{align*}
where $\bar{\cbf}$ denotes the reduced cost. Then we have
\begin{equation}\label{eq: rc}
    \bar{\cbf}^\top_{N^k_\lbf}\left(\xbf^*_{N^k_\lbf} - \lbf_{N^k_\lbf} \right) + \bar{\cbf}^\top_{N^k_\ubf}\left(\xbf^*_{N^k_\ubf} - \ubf_{N^k_\ubf} \right) \leq 0.
\end{equation} 
Consider the most negative term in \eqref{eq: rc} and $\xbf^k$ is not optimal. If $\cbf^\top \xbf^k > \cbf^\top \xbf^*$, \eqref{eq: rc} will become a strict inequality and there exists $j^k\in N^k_\lbf \setminus N^*_\lbf$ such that $\bar{\cbf}_{j^k}\left(\xbf_{j^k}^* - \lbf_{j^k}\right) < 0$ or $j^k\in N^k_\ubf \setminus N^*_\ubf$ such that $\bar{\cbf}_{j^k}\left(\xbf_{j^k}^* - \ubf_{j^k}\right) < 0$. If $\cbf^\top \xbf^k = \cbf^\top \xbf^*$ and each term in \eqref{eq: rc} is 0, we can conduct steepest-edge pivot since we have already reached the optimal face.
\end{proof}

Yang provides similar observation and proof for standard LP formulation \eqref{prob:std-lp} in Remark 3.1 of \cite{yang2020double}, but does not continue to make full use of it. We modify his proof for a more practical formulation of LP and design smart pivot experts based on it.

\paragraph{Designing smart pivot experts}
We design two pivot experts based on two goals and Theorem \ref{thm:idx-in}. Before presenting details, we have to point out that local information is still important. During the development of simplex, much valuable local information has been proposed such as reduced cost and steepest-edge score. Given the optimal basis, our pivot experts combine global information and local information together.

The first pivot expert (Expert I) satisfies the first goal and then tries to satisfy the second. More precisely, it chooses the candidate among the optimal basis with the best steepest-edge score. After the ratio test, it removes the non-basic variable in the optimal basis as far as possible.

The second pivot expert (Expert II) considers the two goals at the same time. It will conduct a ratio test for each candidate in the optimal basis and give preference to those that can remove non-basic variables in the optimal basis if there are any. After candidates are filtered by the two goals, it will choose the one with the best steepest-edge score.

The two experts run at different speeds. Expert I can be calculated at a similar speed as the steepest-edge rule while Expert II runs more slowly and close to the greatest improvement rule (since they both need multiple ratio tests). The two pivot experts share Property \ref{prop:mono} of generating paths with monotone \# DiffOpt defined in Definition \ref{def:diffopt}.  
Several experiments in Section \ref{sec:experiments} will show the superiority of our pivot experts for overall generating shorter paths compared with other classical pivot rules.

\begin{definition}[\# DiffOpt] \label{def:diffopt}
    Let the status vector $\textbf{sta}$ of the vertex $\xbf$ be $\textbf{sta}_i = \begin{cases}
        0, & \text{if non-basic $\xbf_i=\lbf_i$} \\
        1, & \text{if $\xbf_i$ is basic} \\
        2, & \text{if non-basic $\xbf_i=\ubf_i$}
    \end{cases}$. Given the optimal basis, \# DiffOpt is defined as 1-norm of the difference between the status vectors of $\xbf$ and the optimal basis.
\end{definition}

\begin{property}\label{prop:mono}
    Given the optimal basis, before the current objective is optimal, \# DiffOpt is monotonically decreasing.
\end{property}
\begin{proof}
    At each pivot, the basic statuses of at most two variables will change. One enters the basis while one leaves the basis. From Theorem \ref{thm:idx-in}, we know the status of the entering variable is closer to the optimal basis, which will decrease \# DiffOpt by 1, while the leaving variable will increase \# DiffOpt by at most 1. Taken together, the property is proven.
\end{proof}

Here we claim that the two smart pivot experts have independent potential value. They can be the benchmark to estimate the improvement space of the primal simplex method. They can also provide expert labels for other machine learning methods to imitate. If a quick guess of the optimal basis is available by FOMs or machine learning, these experts become practical immediately. A solver developer may also use IPMs to accurately predict the optimal basis and apply our expert rules for simplex analysis and improvement.

At the end of this subsection, we emphasize that the term "expert" refers to overall better performance rather than total transcendence. Recalling the existence of worst cases, it is almost impossible for a pivot rule to completely beat another rule in every single LP instance. Maros \cite{maros2012computational} suggests combining different pivot rules if there are signs of benefit, which is a parallel technical route to ours.

\subsection{Pivot Experts on Klee-Minty Cube Variants}
To further illustrate the value of global information, a linear upper bound is provided for the length of our pivot experts' pivot path on Klee-Minty (KM) cube variants which are usually the worst cases for classical pivot rules.

\paragraph{KM cube variants}
KM cube variants are a well-known class of squashed cubes which usually lead to poor performance of some pivot rules. Example \ref{example:KM} \cite{vanderbei2020linear} and \ref{example:AC} \cite{avis1978notes} encompass a KM variant for Dantzig's rule and the Avis-Chv\'atal polytope for Bland's rule.
\begin{example} \label{example:KM}
    Using Dantzig's rule starting at $\xbf_i^0=\begin{cases}
        0, & \text{for }i=1,\cdots,n \\
        100^{i-n-1}, & \text{for }i=n+1,\cdots,2n
    \end{cases}$, the LP \eqref{prob:KM} requires $2^n - 1$ pivots to reach the optimal solution.
    \begin{equation} \label{prob:KM}
      \begin{aligned}
          \min \  & \sum_{j=1}^n 10^{n-j} \xbf_j                                         \\
          \st\    & 2 \sum_{j=1}^{i-1} \xbf_j + \xbf_i + \xbf_{n+i} = 100^{i-1},\ i=1,\cdots,n \\
                  & \xbf_i \geq 0,\ i=1,\cdots,2n.
      \end{aligned}
    \end{equation}
\end{example}

\begin{example} \label{example:AC}
    Using Bland's rule starting at $\xbf_i^0=\begin{cases}
        0, & \text{for }i=1,\cdots,n \\
        1, & \text{for }i=n+1,\cdots,2n
    \end{cases}$, the pivot number for solving the LP \eqref{prob:AC} with $\epsilon\in(0,0.5)$ pivots is bounded from below by the $n$th Fibonacci number.
    \begin{equation} \label{prob:AC}
      \begin{aligned}
          \min \  & \sum_{j=1}^n \epsilon^{n-j} \xbf_j                                          \\
          \st\    & 2 \sum_{j=1}^{i-1} \epsilon^{i-j} \xbf_j + \xbf_i + \xbf_{n+i} = 1,\ i=1,\cdots,n \\
                  & \xbf_i \geq 0,\ i=1,\cdots,2n.
      \end{aligned}
    \end{equation}
\end{example}

These KM cubes share some similar properties. First, the feasibility set is combinatorially equivalent to the standard $n$-dimensional cube $\Ccal_n = \{(\xbf, \ybf)\in\Rbb^n\times\Rbb^n: \xbf + \ybf = \onebf,\ \xbf, \ybf \geq \zerobf\}$, which means there exists a one-to-one correspondence between their faces. Second, each vertex is non-degenerate. The standard cube $\Ccal_n$ here is obtained via adding slacks $\ybf$ in the cube $[0,1]^n$. $\Ccal_n$ has $2^n$ vertices whose first $n$ elements are $\xbf\in\{0,1\}^n$ and $\ybf = \onebf - \xbf$.

\paragraph{The experts' linear upper bound on KM cubes}
The main idea for deriving a linear upper bound for the pivot experts on KM cubes is to analyze the length of paths with monotone \# DiffOpt. The proof process is divided into three steps. In essence, we start by bounding path lengths on the basic cube $\Ccal_n$, extend that to polytopes combinatorially equivalent to $\Ccal_n$, and then apply it to the pivot experts in the KM setting under certain mild assumptions. This allows us to derive an overall linear upper bound on KM cubes.

Lemma \ref{lemma:diffopt-box} describes an efficient way to calculate \# DiffOpt on $\Ccal_n$.
\begin{lemma} \label{lemma:diffopt-box}
    For $\Ccal_n$ with any optimal basis $\Bbf^*$ and the corresponding optimal solution $(\xbf^*,\ybf^*)$, \# DiffOpt of any vertex $(\xbf,\ybf)$ is equal to $\|(\xbf,\ybf) - (\xbf^*,\ybf^*)\|_1$.
\end{lemma}
\begin{proof}
    For any vertex $(\xbf,\ybf) \in \{0, 1\}^{2n}$, note the basic variable 1 and non-basic variable 0 and we can get the vector $\vbf\in\{0,1\}^{2n}$. It is easy to verify that $\vbf = (\xbf, \ybf)$. Thus, $\text{\# DiffOpt} = \|\vbf - \vbf^*\|_1 = \|(\xbf,\ybf) - (\xbf^*,\ybf^*)\|_1$.
\end{proof}

Combining Lemma \ref{lemma:diffopt-box} and the algebraic structure of vertices of $\Ccal_n$, we have Corollary \ref{col:diffopt-2n} and \ref{col:diffopt-adj-2}. Then we can derive an upper bound of $n$ for the length of all paths with monotone \# DiffOpt in Theorem \ref{thm:path-cube-n}.
\begin{corollary} \label{col:diffopt-2n}
    For $\Ccal_n$ with any optimal basis $\Bbf^*$, \# DiffOpt of each vertex is upper bounded by $2n$.
\end{corollary}
\begin{proof}
    Since the elements in each vertex of $\Ccal_n$ are either 0 or 1, from Lemma \ref{lemma:diffopt-box} we have
    $$\text{\# DiffOpt} = \sum_{i=1}^n |\xbf_i - \xbf^*_i| + \sum_{i=1}^n |\ybf_i - \ybf^*_i| \leq \sum_{i=1}^n |1 - 0| + \sum_{i=1}^n |1 - 0| = 2n.$$
\end{proof}

\begin{corollary} \label{col:diffopt-adj-2}
    For each pair of adjacent vertices in $\Ccal_n$, the difference in \# DiffOpt is 2.
\end{corollary}
\begin{proof}
    Notice that adjacent vertices $\xbf^1$ and $\xbf^2$ of the basic cube $[0, 1]^n$ will only differ in one dimension. For the corresponding pair of adjacent vertices $(\xbf^1,\ybf^1)$ and $(\xbf^2,\ybf^2)$ of $\Ccal_n$, they will have two different components since both of them satisfy $\xbf + \ybf = \onebf$, say $\xbf_i$ and $\ybf_i$. Therefore, we have $\xbf^1_i = \ybf^2_i, \xbf^1_i + \xbf^2_i = 1$, and the difference in \# DiffOpt is    
    \begin{equation*}
        \begin{aligned}
            & \left| \text{\# DiffOpt of $(\xbf^1,\ybf^1)$} - \text{\# DiffOpt of $(\xbf^2,\ybf^2)$} \right| \\
            =\ & \left| (|\xbf^1_i - \xbf^*_i| + |\ybf^1_i - \ybf^*_i|) - (|\xbf^2_i - \xbf^*_i| + |\ybf^2_i - \ybf^*_i|) \right| \\
            =\ & \left| (|\xbf^1_i - \xbf^*_i| + |1-\xbf^1_i - \ybf^*_i|) - (|1-\xbf^1_i - \xbf^*_i| + |\xbf^1_i - \ybf^*_i|) \right| \\
            =\ & \left| (|\xbf^1_i - \xbf^*_i| + |\xbf^1_i - \xbf^*_i|) - (|\xbf^1_i - \ybf^*_i| + |\xbf^1_i - \ybf^*_i|) \right| \\
            =\ & 2\left| |\xbf^1_i - \xbf^*_i| - |\xbf^1_i - \ybf^*_i| \right| \\
            =\ & 2\left| |\xbf^1_i - \xbf^*_i| - |\xbf^1_i - (1 - \xbf^*_i)| \right| = 2.
        \end{aligned}
    \end{equation*}
    The last equation is because $\xbf^1_i, \xbf^*_i \in \{0,1\}$.
\end{proof}

\begin{theorem} \label{thm:path-cube-n}
    For $\Ccal_n$ with initial point $(\xbf^0,\ybf^0)$ and any optimal basis $\Bbf^*$, the length of the path with monotone \# DiffOpt is $\frac{\text{\# DiffOpt of $(\xbf^0,\ybf^0)$}}{2}$, which is bounded by $n$ from above.
\end{theorem}
\begin{proof}
    From Corollary \ref{col:diffopt-adj-2}, we know that a path with monotone \# DiffOpt will have a decrease of 2 in \# DiffOpt at every pivot. Thus, the length is $\frac{\text{\# DiffOpt of $(\xbf^0,\ybf^0)$}}{2} \leq n$. The inequality comes from Corollary \ref{col:diffopt-2n}.
\end{proof}

After analyzing the paths on standard cubes, we turn to our pivot experts and paths on some more general polytopes which cover the KM cube cases and have Theorem \ref{thm:path-KM-n} and \ref{thm:exp-KM-n}.
\begin{theorem} \label{thm:path-KM-n}
    For any polytope combinatorially equivalent to $\Ccal_n$ with non-degenerate vertices, initial point $(\xbf^0,\ybf^0)$, and any optimal basis $\Bbf^*$, the length of path with monotone \# DiffOpt is $\frac{\text{\# DiffOpt of $(\xbf^0,\ybf^0)$}}{2}$, which is bounded by $n$ from above.
\end{theorem}
\begin{proof}
    For any polytope combinatorially equivalent to $\Ccal_n$, there is a one-to-one correspondence between faces of the polytope and $\Ccal_n$, and the assumption of non-degeneracy guarantees that each vertex of the polytope has a unique \# DiffOpt. Therefore, we can obtain similar results as in Theorem \ref{thm:path-cube-n}.
\end{proof}

\begin{theorem} \label{thm:exp-KM-n}
    For any polytope combinatorially equivalent to $\Ccal_n$ with non-degenerate vertices, initial point $(\xbf^0,\ybf^0)$, and the single optimal basis $\Bbf^*$, the length of a path generated by our pivot experts is upper bounded by $\frac{\text{\# DiffOpt of $(\xbf^0,\ybf^0)$}}{2}$, which is bounded by $n$ from above.
\end{theorem}
\begin{proof}
    The assumption of the single optimal basis guarantees our pivot experts to find the entering variable at each iteration. Otherwise, there may not exist such a variable in Theorem \ref{thm:idx-in} if we have found an optimal solution but not an optimal basis. Since our experts' pivot path has monotone \# DiffOpt, we can obtain the same linear upper bound as in Theorem \ref{thm:path-KM-n}.
\end{proof}

Theorem \ref{thm:exp-KM-n} illustrates the value of global information. With the guidance of the given optimal basis, our experts will avoid being led to the worst by misleading local information on various KM variants. Notice that the upper bound holds with arbitrary local information including reduced cost, steepest-edge score, or even random selection. In this aspect, the monotone \# DiffOpt is more like a combinatorial property rather than an algebraic property as the monotone objective value. The strong performance of our experts on KM variants will not be negatively impacted by scaling, which is different from most other classical pivot rules.

\section{Learning as a Pivot Expert}\label{sec:learn}

Despite the superiority of our pivot experts, their requirement for the optimal basis seems a critical defect preventing us from direct application. In this section, we proceed with removing the need for that additional global information via machine learning. To be more specific, we encode the LP as a bipartite graph and build a GCNN model to imitate the pivot expert's choices, which is similar to Gasse's GCNN \cite{gasse2019exact}.

\paragraph{State encoding and input features}
The primal simplex method can be viewed as a Markov decision process, as shown in Figure \ref{fig:simplex}. At the $k$th iteration, the state $\sbf_k$ contains the LP instance and the current basic feasible solution $\xbf^k$. The action space $\Acal(\sbf_k)$ includes all directed edges that can improve the objective value. An action $\abf_k$ is selected according to a certain pivot rule and $\xbf^k$ will move along $\abf_k$ until reaching the next vertex $\xbf^{k+1}$.
\begin{figure}[!ht]
    \centering
    \includegraphics[scale=0.4]{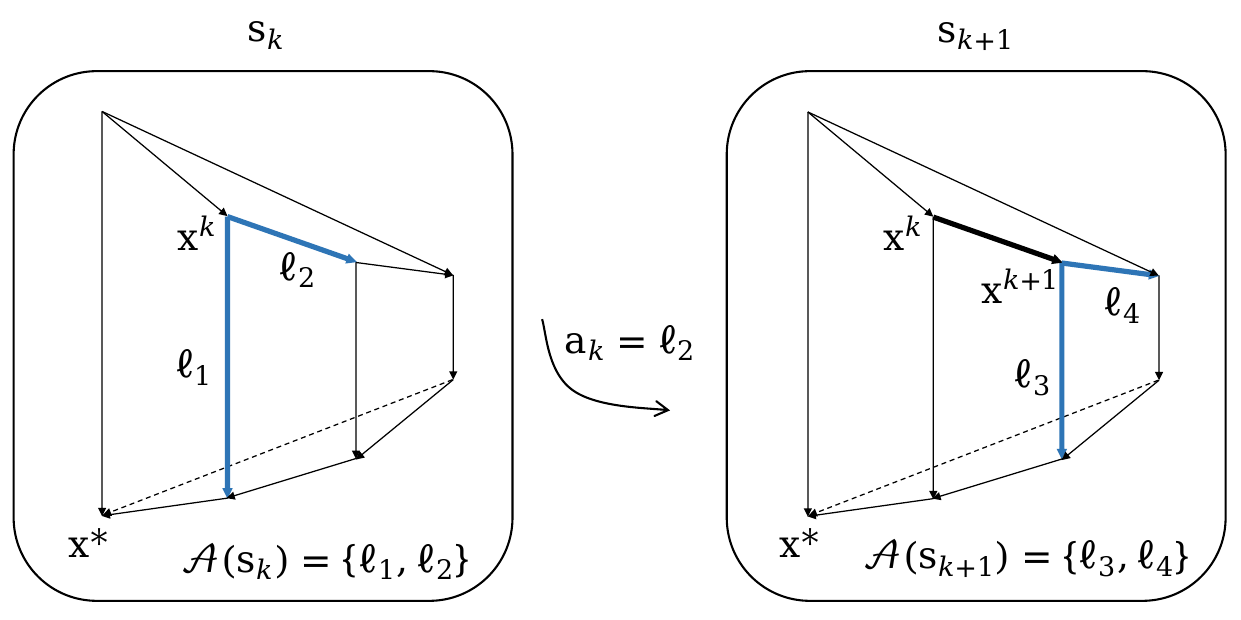}
    \caption{The primal simplex method can be viewed as a Markov decision process. The LP is depicted as a polyhedron with directions that can improve objective value marked on edges. $\xbf^*$ denotes the optimal solution.}
    \label{fig:simplex}
\end{figure}

We encode the state of LP and current solution into a bipartite graph, see Figure \ref{fig:bigraph}.
\begin{figure}[!ht]
    \centering
    \includegraphics[scale=0.6]{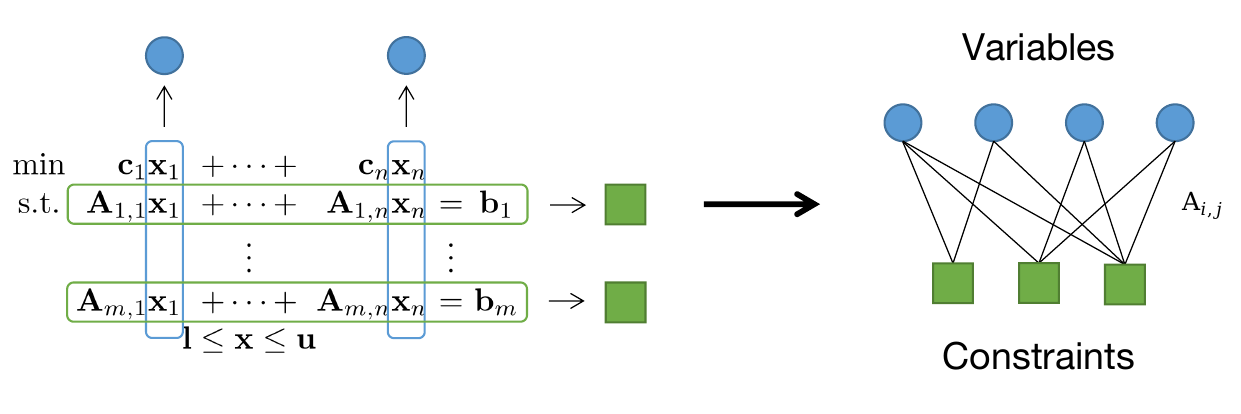}
    \caption{Bipartite graph encoding for LP.}
    \label{fig:bigraph}
\end{figure}
Variables and constraints form two classes of nodes which will be linked by an edge if the corresponding coefficient $\Abf_{i,j}$ is non-zero. Each node and edge carry some features which are listed in Appendix \ref{apx:feat} (Table \ref{tab:feat}). The features are picked for pivot decision and thus slightly differ from those defined by Gasse.

\paragraph{Policy for imitating experts}
We input the bipartite graph to a GCNN model, after which follows a filter and a softmax function. The filter selects all candidates to enter the basis and the softmax function provides their probability to enter the basis. Our GCNN model shares a similar structure with Gasse's but has slightly wider and deeper networks.

It is worth highlighting the differences between learning to pivot for LP and learning to branch for MIP although they both have to select variables. Branching makes choices of variables on leaf nodes of the BnB tree while pivoting runs on a single path. When selecting a variable, all fractional variables with the integer constraints at each leaf node are potential candidates according to the node selection strategy, but the primal simplex method only allows variables that can improve the objective value to enter the basis. Considering the single path and candidate restriction, a bad pivot choice may be far more uncontrollable than a bad branching choice.

\paragraph{Summary}
The pivot expert designing and learning framework can be summarized in Figure \ref{fig:framework}. For a class of LP instances, which are expected to share common features in polyhedra, we have designed two pivot experts that can generate shorter paths. In order to clone the experts using imitation learning, we collect data pairs of encoded LP instances and labels containing bipartite graphs and corresponding experts' pivot choices. With these data, the policy network is trained to replicate the experts' behavior accurately. The high accuracy will help the learned pivot rule correct directions even if a few wrong steps are taken previously.
\begin{figure}[!ht]
    \centering
    \subcaptionbox{A class of LP instances.}{
        \includegraphics[scale=0.35]{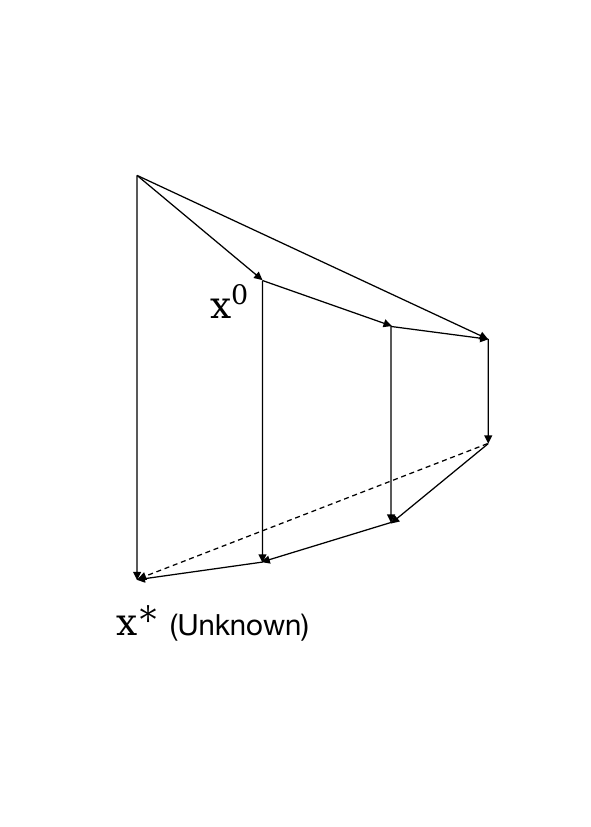}
    }
    \hspace{1em}
    \subcaptionbox{Given the optimal basis, pivot experts can generate shorter paths.}{
        \includegraphics[scale=0.35]{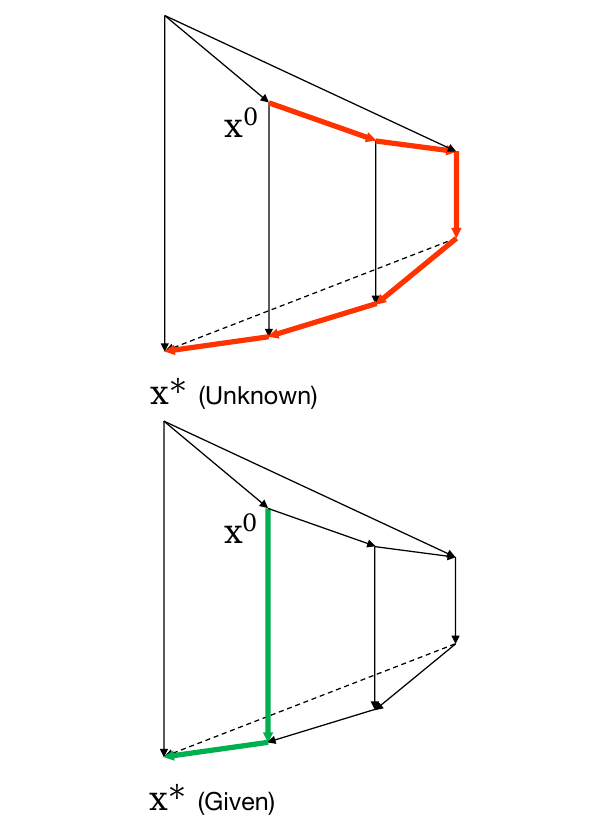}
    }
    \hspace{1em}
    \subcaptionbox{Collect expert choices for imitation.}{
        \includegraphics[scale=0.35]{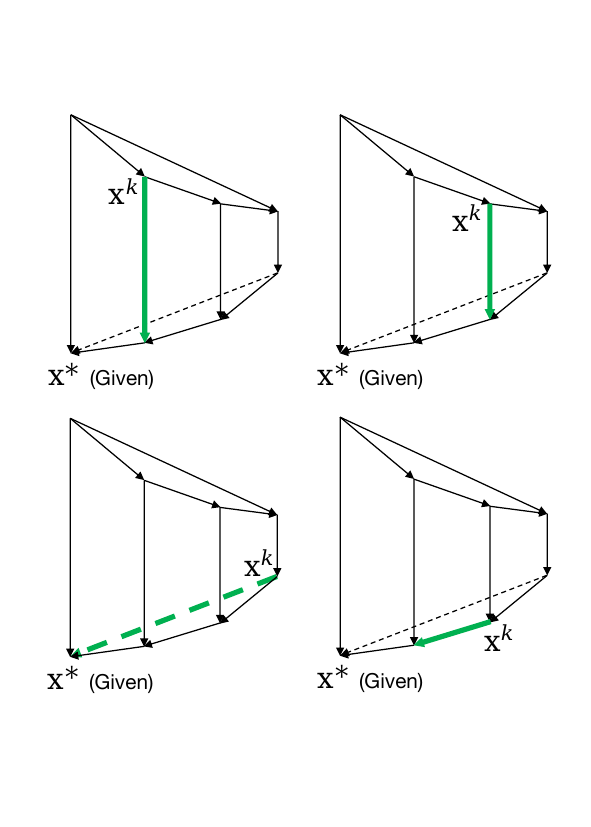}
    }
    \hspace{1em}
    \subcaptionbox{Learned pivot rule gains improvement on new LP instances.}{
        \includegraphics[scale=0.35]{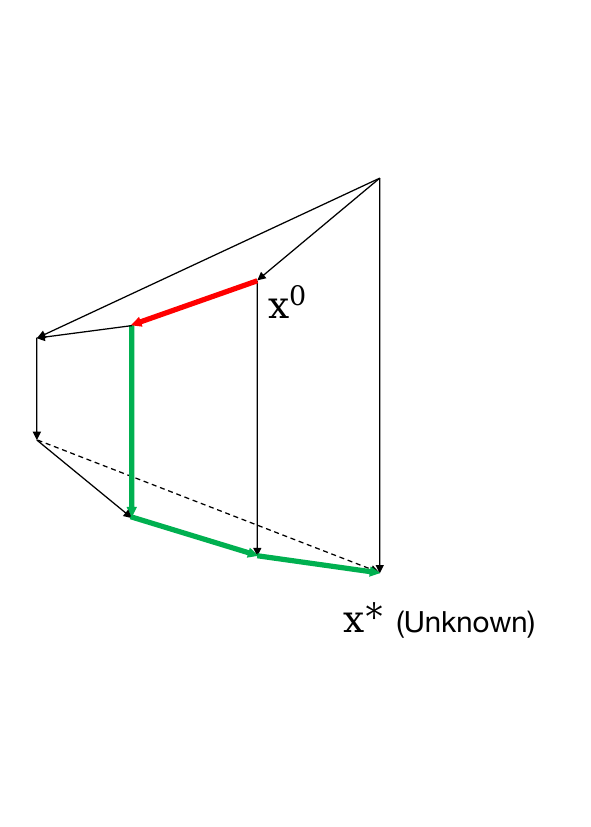}
    }
    \caption{Pivot expert designing and learning framework in the order of (a) \textrightarrow (b) \textrightarrow (c) \textrightarrow (d). Bold {\color{darkgreen}green} lines are expert choices while bold {\color{red}red} lines are bad choices.}
    \label{fig:framework}
\end{figure}

\section{Experiments}\label{sec:experiments}

Our experiments are twofold. In Section \ref{subsec:exp}, we contrast our advanced pivot experts against others, demonstrating their ability to produce shorter pivot paths. Section \ref{subsec:learn} showcases the commendable performance of the learned rule, affirming the feasibility of training our expert rules via imitation learning without sacrificing overall enhancement. All tests are conducted on an AMD Ryzen 7 5700X CPU and NVIDIA RTX 3060 12GB GPU.

\paragraph{Pivot rules to be compared}
Table \ref{tab:pivotring-rules} enumerates the pivot rules under comparison. We adopt five conventional pivot rules (Bland's, Dantzig's, SE, GI, and LD) as our benchmarks. Our primary focus is on two expert rules (EXP and EXP-II), alongside our learned rule (EXP-LEARN). In all experiments, every pivot rule receives a consistent initial basis from Phase I, resolved by SE. Notably, EXP and EXP-II, requiring an additional optimal basis, are provided in Phase II by SE. EXP-LEARN operates without the extra information that classical rules might need. NO-LOCAL, a derivative of our pivot expert, omits local information and serves for ablation analysis.

\begin{table}[!ht]
    \centering
    \begin{tabular}{c|c|c}
        \toprule
        Type & Notation & Pivot rule \\
        \midrule
        \multirow{5}{*}{Classical rules} & Bland & Bland's rule \\
         & Dantzig & Dantzig's rule \\
         & SE & Steepest-edge rule \\
         & GI & Greatest improvement rule \\
         & LD & Largest distance rule \\
        \hline
        \multirow{3}{*}{Our experts} & EXP & Expert I \\
         & EXP-II & Expert II \\
         & EXP-LEARN & Rule imitating Expert I \\
        \hline
        Ablation study & NO-LOCAL & Expert without local information \\
        \bottomrule
    \end{tabular}
    \caption{Pivot rules to be compared.}
    \label{tab:pivotring-rules}
\end{table}

\paragraph{Benchmarks}
We have chosen a diverse set of LP problem tests, including a NETLIB subset \cite{gay1985electronic} and LP relaxations from four combinatorial optimization (CO) classes. NETLIB, a standard LP benchmark, offers varied LP instances in both scale and structure. Our CO classes cover set covering (SC), combinatorial auction (CA), capacitated facility location (FL), and maximum independent set (IS). These CO problems, inspired by \cite{gasse2019exact}, may differ in scale. We presolve NETLIB instances with Gurobi 10.0.2 and CO instances with SCIP 8.0.3.

\paragraph{Evaluation}
We evaluate pivot rule performance primarily using pivot path length, with the geometric mean of pivot lengths serving as our benchmark metric. Two reasons drive this choice. First, while our Python-implemented pivot rules might not mirror modern solvers' speed, they are generally comparable in pivot numbers with Gurobi's primal simplex for many LP instances. Second, pivot path length better captures the simplex method's complexity. Additionally, for thoroughness, we will also report each rule's execution time.

\subsection{Testing Pivot Experts}\label{subsec:exp}
We evaluate our expert rules against classical methods using the NETLIB subset and various optimization benchmarks.

\subsubsection{NETLIB}
\paragraph{Setup}

We utilize 77 selected NETLIB instances, optimized for time and to bypass numerical issues. Each instance adheres to a 300-second time limit. The number of constraints $m$ and variables $n$ for these instances are provided in Table \ref{tab:netlib-scale-exp}.

\begin{table}[!ht]\small
\centering
\begin{tabular}{c|ccc}
\toprule
& \# instance & $m$ & $n$\\
\midrule
NETLIB & 77 & 356.56 ($\pm$ 364.37) & 1439.61 ($\pm$ 2975.50) \\
\bottomrule
\end{tabular}
\caption{Average scales of presolved NETLIB instances.\protect\footnotemark}
\label{tab:netlib-scale-exp}
\end{table}

\footnotetext{Standard deviations are provided in parentheses. This convention is maintained for subsequent tables.}

\paragraph{Numerical results}

The results in Tables \ref{tab:netlib-iter-win} and \ref{tab:netlib-iter-gm} emphasize the efficacy of EXP and EXP-II over classical pivot rules.

\begin{table}[!ht]\small
\centering
\begin{tabular}{ccccc|c}
\toprule
Bland & Dantzig & SE & GI & LD & EXP \\
\midrule
2 & 7 & 16 & 15 & 6 & \textbf{52} \\
\bottomrule
\end{tabular}
\hspace{1em}
\begin{tabular}{ccccc|c}
\toprule
Bland & Dantzig & SE & GI & LD & EXP-II \\
\midrule
2 & 6 & 17 & 12 & 6 & \textbf{54} \\
\bottomrule
\end{tabular}
\caption{Number of wins on the NETLIB subset with a total of 77 instances.\protect\footnotemark}
\label{tab:netlib-iter-win}
\end{table}
\footnotetext{Bolded results indicate the best performance among the evaluated pivot rules. This convention is maintained for subsequent tables.}

Table \ref{tab:netlib-iter-gm} showcases the geometric mean of pivot numbers on the NETLIB subset. It is evident that EXP and EXP-II outperform, while Bland's rule has the worst performance and is thus excluded in future tests.


\begin{table}[!ht]
    \centering
    \begin{tabular}{ccccc|cc|c}
         \toprule
         Bland & Dantzig & SE & GI & LD & EXP & EXP-II & NO-LOCAL\\
         \midrule
         850.48 & 198.15 & 121.28 & 130.57 & 176.51 & \textbf{111.94} & 118.28  & 139.03\\
         \bottomrule
    \end{tabular}
    \caption{Geometric mean of pivot numbers on the NETLIB subset.}
    \label{tab:netlib-iter-gm}
\end{table}

\paragraph{Ablation study: the role of local information in expert rules}

Classical pivot rules, like the SE, GI, Bland's, and Dantzig's rules, largely utilize local information, such as reduced costs. In contrast, our expert rules merge both global (the optimal basis) and local information (the steepest-edge score) to set the pivot direction. Here, we emphasize the pivotal role local information plays in optimizing expert rules.

We introduce a variant of the EXP rule, called NO-LOCAL, that omits local information. In this model, the entering variable is randomly chosen from candidates in the optimal basis. Its efficacy is tested on the NETLIB subset.


\begin{table}[!ht]
    \centering
    \begin{tabular}{ccccc|c|c}
    \toprule
        Bland & Dantzig & SE & GI & LD & EXP & NO-LOCAL \\
    \midrule
        2 & 6 & 16 & 15 & 6 & \textbf{52} & 4\\
    \bottomrule
    \end{tabular}
    \caption{Number of wins on the NETLIB subset with a total of 77 instances.}
    \label{tab:netlib-wins-exp}
\end{table}

\begin{table}[!ht]
    \centering
    \begin{tabular}{ccccc|c|c}
    \toprule
        Bland & Dantzig & SE & GI & LD & EXP-II & NO-LOCAL \\
    \midrule
        2 & 6 & 17 & 12 & 6 & \textbf{54} & 5\\
    \bottomrule
    \end{tabular}
    \caption{Number of wins on the NETLIB subset with a total of 77 instances.}
    \label{tab:netlib-wins-exp-inout}
\end{table}

Table \ref{tab:netlib-iter-gm} and its counterparts reveal that NO-LOCAL underperforms against EXP and EXP-II rules and even lags behind the classical SE rule. The results underscore the diminished efficacy of the EXP rule when local insights are absent, leading to increased pivot numbers and fewer wins. Local information is evidently instrumental in optimizing pivot decisions.

\subsubsection{Combinatorial Optimization Benchmarks} \label{subsubsec:co-bench}
\paragraph{Setup}

Based on the guidelines from \cite{gasse2019exact}, we generate 100 random instances for each combinatorial benchmark. For the set covering problems, instances have 400 columns and 200 rows. Combinatorial auction problems have 100 items and 500 bids. For capacitated facility location problems, instances contain 20 facilities and 15 customers. Finally, maximum independent set problems have 150 nodes with an affinity value set to 2. Table \ref{tab:co-scale-exp} details the scales of these presolved CO instances.

\begin{table}[!ht]\small
    \centering
    \begin{tabular}{c|ccc}
    \toprule
      & \# instance & $m$ & $n$\\
    \midrule
    \textbf{SC} & 100 & 200.00 ($\pm$ 0.00) & 400.00 ($\pm$ 0.00) \\
    \textbf{CA} & 100 & 181.46 ($\pm$ 4.78) & 425.09 ($\pm$ 16.62) \\
    \textbf{FL} & 100 & 336.00 ($\pm$ 0.00) & 315.00 ($\pm$ 0.00) \\
    \textbf{IS} & 100 & 290.18 ($\pm$ 2.12) & 150.00 ($\pm$ 0.00) \\
    \bottomrule
    \end{tabular}
    \caption{Average scales of presolved CO instances.}
    \label{tab:co-scale-exp}
\end{table}

\paragraph{Numerical results}

\begin{table}[!ht]
    \centering
    \begin{tabular}{c|cccc|cc}
    \toprule
              & Dantzig & SE     & GI      & LD     & EXP    & EXP-II  \\
    \midrule
    \textbf{SC}  & 2093.30  & 417.71 & 867.47 & 456.15 & \textbf{267.66} & 282.03     \\ \hline
    \textbf{CA} & 856.35 & 265.74 & 3356.4 & 278.86 & 115.02 & \textbf{113.06}     \\ \hline
    \textbf{FL}    & 461.81  & 256.50 & 311.85  & 412.02 & 242.32 & \textbf{235.19}     \\ \hline
    \textbf{IS}    & 334.35 & 114.03 & 345.61 & 114.03 & 113.14 & \textbf{113.02}     \\
    \bottomrule
    \end{tabular}
    \caption{Geometric mean of pivot numbers on the CO problem sets.}
    \label{tab:learn-train-co-iter}
\end{table}

Table \ref{tab:learn-train-co-iter} reveals that among classical pivot rules, the SE rule performs best. However, our designed rules, EXP and EXP-II, surpass all classical rules in performance, indicating the potential of these newer strategies. Specifically, EXP excels on the SC benchmark, while the EXP-II rule consistently outshines all classical methods on the CA, FL, and IS benchmarks. It is worth noting that Dantzig's rule often lags behind, and the GI rule particularly struggles with the CA benchmark.

\begin{table}[!ht]
    \centering
    \begin{tabular}{c|cccc|cc}
    \toprule
              & Dantzig & SE     & GI      & LD     & EXP    & EXP-II  \\
    \midrule
    \textbf{SC}  & 0 & 2 & 0 & 0 & \textbf{63} & 37    \\
    \hline
    \textbf{CA} & 0 & 0	& 0 & 0	& 27 & \textbf{77}    \\
    \hline
    \textbf{FL}    & 0 & 14 & 5  & 3 & 40 & \textbf{42}    \\
    \hline
    \textbf{IS}    & 0 & 39 & 0 & 39 & 91 & \textbf{100}    \\
    \bottomrule
    \end{tabular}
    \caption{Number of wins on the CO problem sets with a total of 100 instances.}
    \label{tab:learn-train-co-iter-win-nolearn}
\end{table}

As for the number of wins in Table \ref{tab:learn-train-co-iter-win-nolearn}, our designed rules outperform the classical rules by a large margin. Besides, the noteworthy consistency of EXP and EXP-II across benchmarks, as similarly showcased in Table \ref{tab:learn-train-co-iter}, underscores their robustness. This uniform performance indicates that their dominance is not merely attributed to specific outlier instances but represents wider applicability and efficacy.

\begin{table}[!ht]
    \centering
    \begin{tabular}{c|cccc|cc}
    \toprule
              & Dantzig & SE     & GI      & LD     & EXP    & EXP-II  \\
    \midrule
    \textbf{SC}  & 1.243 & 0.313 & 5.374 & 0.332 & \textbf{0.200} & 0.683     \\
    \hline
    \textbf{CA} & 0.504 & 0.202 & 23.93 & 0.207 & \textbf{0.085} & 0.185      \\
    \hline
    \textbf{FL}    & 0.473 & 0.317 & 1.434  & 0.463 & \textbf{0.299} & 0.769     \\
    \hline
    \textbf{IS}    & 0.295 & 0.119 & 0.719  & \textbf{0.117} & 0.127 & 0.282     \\
    \bottomrule
    \end{tabular}
    \caption{Geometric mean of solving time (in seconds) on the CO problem sets.}
    \label{tab:learn-train-co-time}
\end{table}

Table \ref{tab:learn-train-co-time} showcases the efficacy of EXP and EXP-II over classical pivot rules regarding time. Notably, EXP stands out with the shortest average time across all datasets, highlighting its speedy convergence. Though the EXP-II rule efficiently minimizes iterations, its computational overhead from the additional tie-breaking step compromises its time efficiency, especially when compared to EXP. Summarily, considering both time and iteration counts, EXP reigns as the prime pivot rule for these problem sets in time efficiency.

While our designed experts have shown promising performance across several benchmarks, their complex and non-classical nature might make them difficult to integrate directly into existing simplex solvers or optimization frameworks. In the following Section \ref{subsec:learn}, we present an imitation learning strategy to address these implementation hurdles.

\subsection{Testing Learned Pivot Rule}\label{subsec:learn}

We conduct experiments with EXP-LEARN on combinatorial optimization benchmarks, employing the imitation learning approach to emulate the EXP rule.

\subsubsection{Combinatorial Optimization Benchmarks} \label{subsubsec:co-learn-bench}

\paragraph{Setup}

The benchmark scale mirrors that of the prior CO experiment (Section \ref{subsubsec:co-bench}). Scales of presolved CO instances are shown in Table \ref{tab:co-scale-learn}.

\begin{table}[!ht]\small
    \centering
    \resizebox{\textwidth}{!}{%
    \begin{tabular}{c|ccc|ccc|ccc}
    \toprule
     & \# train & $m$ & $n$ & \# valid & $m$ & $n$ & \# test & $m$ & $n$ \\
    \midrule
      \textbf{SC}  & 1000 & 200.00 ($\pm$ 0.00) & 400.00 ($\pm$ 0.00) & 200 & 200.00 ($\pm$ 0.00) & 400.00 ($\pm$ 0.00) & 200 & 200.00 ($\pm$ 0.00) & 400.00 ($\pm$ 0.00)  \\
      \textbf{CA} & 1000 & 181.93 ($\pm$ 5.02) & 427.42 ($\pm$ 16.98) & 200 & 181.98 ($\pm$ 5.45) & 428.15 ($\pm$ 18.7) & 200 & 182.48 ($\pm$ 5.53) & 427.76 ($\pm$ 18.56) \\
      \textbf{FL}    & 1000 & 336.00 ($\pm$ 0.00) & 315.00 ($\pm$ 0.00) & 200 & 336.00 ($\pm$ 0.00) & 315.00 ($\pm$ 0.00) & 200 & 336.00 ($\pm$ 0.00) & 315.00 ($\pm$ 0.00)  \\
      \textbf{IS}    & 1000 & 290.20 ($\pm$ 2.37) & 150.00 ($\pm$ 0.00) & 200 & 290.18 ($\pm$ 2.27) & 150.00 ($\pm$ 0.00) & 200 & 290.19 ($\pm$ 2.09) & 150.00 ($\pm$ 0.00) \\
    \bottomrule
    \end{tabular}%
    }
    \caption{Average scales of presolved CO instances for EXP-LEARN to imitate EXP.}
    \label{tab:co-scale-learn}
\end{table}

\paragraph{Training procedure} 
During training, we use 5 unique seeds for data generation and model training. For each CO benchmark, identical hyperparameters are employed, with 50,000 pivot samples sourced from 1,000 instances for training and 10,000 from 200 instances for validation.

\paragraph{Performance metrics}
Performance is assessed using the GCNN model's validation accuracy, detailed in Table \ref{tab:learn-valid-acc}. We rely on Top 1, Top 3, and Top 5 accuracies. To gauge its real-world relevance, the model is tested on 200 fresh instances, emphasizing its generalization across benchmarks. Results can be found in Tables \ref{tab:learn-test-co-iter} to \ref{tab:learn-test-co-time}.

\paragraph{Numerical results}

\begin{table}[!ht]
    \centering
    \begin{tabular}{c|c|c|c}
    \toprule
              & Top 1 Acc & Top 3 Acc & Top 5 Acc \\
    \midrule
    \textbf{SC}  & 0.533 ($\pm$ 0.004)    & 0.847 ($\pm$ 0.005)    & 0.927 ($\pm$ 0.003)    \\ \hline
    \textbf{CA} & 0.362 ($\pm$ 0.004)   & 0.656 ($\pm$ 0.002)    & 0.784 ($\pm$ 0.002)    \\ \hline
    \textbf{FL}    & 0.499 ($\pm$ 0.007)    & 0.776 ($\pm$ 0.007)    & 0.870 ($\pm$ 0.005)    \\ \hline
    \textbf{IS}    & 0.257 ($\pm$ 0.002)   & 0.420 ($\pm$ 0.002)    & 0.511 ($\pm$ 0.001)    \\
    \bottomrule
    \end{tabular}
    \caption{Accuracies on the CO validation sets.}
    \label{tab:learn-valid-acc}
\end{table}

As observed in Table \ref{tab:learn-valid-acc}, our model consistently achieves a Top 1 accuracy over 25\% for all problems. This suggests a greater than 25\% chance of choosing the expert action from potential pivot candidates. Notably, the simplex method's nature in LP means that while Top 3 and Top 5 accuracies are significant, they cannot be directly applied, marking a limitation. However, it needs to be emphasized that comparing validation accuracy across different benchmarks is not meaningful. Moreover, we choose not to include test accuracy, as our primary focus is not on direct accuracy comparison but on the pivotal metric of pivot length. 


\begin{table}[!ht]
    \centering
    \begin{tabular}{c|ccc|c|cc}
    \toprule
              & SE     & GI      & LD     & EXP-LEARN & EXP    & EXP-II  \\
    \midrule
    \textbf{SC}  & 418.51 & 989.53 & 468.42 & \textbf{336.24} ($\pm$ 5.72) & 268.33 & 280.31    \\ \hline
    \textbf{CA} & 265.61 & 1339.67 & 275.88 & \textbf{223.33} ($\pm$ 3.30) & 114.65 & 111.70     \\ \hline
    \textbf{FL}    & 241.74 & 303.48  & 376.80 & \textbf{239.01} ($\pm$ 1.73) & 223.57 & 226.98     \\ \hline
    \textbf{IS}    & 113.52 & 301.79 & 113.52 & \textbf{113.44} ($\pm$ 0.01) & 112.68 & 112.61     \\
    \bottomrule
    \end{tabular}
    \caption{Geometric mean of pivot numbers on the CO test sets.\protect\footnotemark}
    \label{tab:learn-test-co-iter}
\end{table}

\footnotetext{The results for EXP and EXP-II are for reference only as they cannot be directly used in practical applications.}

Table \ref{tab:learn-test-co-iter} shows that EXP-LEARN consistently outshines its competitors, highlighting its superior efficiency. While EXP-LEARN displays great performance among most benchmarks, it exhibits only a slight lead over the SE rule in certain benchmarks, primarily due to the foundational efficiency of the EXP rule it is built upon. This demonstrates that while imitation learning brings benefits, it does not fully bridge the performance gap between SE and EXP.

\begin{table}[!ht]
    \centering
    \begin{tabular}{c|ccc|c}
    \toprule
              & SE     & GI      & LD     & EXP-LEARN \\
    \midrule
    \textbf{SC}  & 8.20 ($\pm$ 3.06) & 0.00 ($\pm$ 0.00) & 0.20 ($\pm$ 0.40) & \textbf{192.00} ($\pm$ 2.68)    \\
    \hline
    \textbf{CA} & 7.00 ($\pm$ 2.00)	& 0.00 ($\pm$ 0.00) & 1.60 ($\pm$ 1.85)	& \textbf{192.40} ($\pm$ 3.01)    \\
    \hline
    \textbf{FL}    & 80.60 ($\pm$ 1.96)	& 17.40 ($\pm$ 2.24) & 17.40 ($\pm$ 2.15)	& \textbf{102.80} ($\pm$ 2.93)    \\
    \hline
    \textbf{IS}    & 170.80 ($\pm$ 1.33)	& 0.00 ($\pm$ 0.00) & 170.80 ($\pm$ 1.33)	& \textbf{186.40} ($\pm$ 1.50)    \\
    \bottomrule
    \end{tabular}
    \caption{Number of wins on the CO test sets with a total of 200 instances.}
    \label{tab:learn-test-co-iter-win}
\end{table}

Table \ref{tab:learn-test-co-iter-win} underscores the consistency of the EXP-LEARN rule. Its dominant performance is not due to a few outliers but is maintained across numerous instances in each benchmark. This indicates a robust and generalizable model, proving EXP-LEARN's reliability in varied scenarios and highlighting its potential for broader applications in optimization tasks.

\begin{table}[!ht]
    \centering
    \begin{tabular}{c|ccc|c|cc}
    \toprule
              & SE    & GI     & LD   & EXP-LEARN & EXP   & EXP-II  \\
    \midrule
    \textbf{SC}  & 0.537 & 6.849 & \textbf{0.317} & 1.359 ($\pm$ 0.022) & 0.317 & 0.827    \\
    \hline
    \textbf{CA} & 0.210 & 8.438 & \textbf{0.202} & 0.806 ($\pm$ 0.011) & 0.083 & 0.178    \\
    \hline
    \textbf{FL}    & \textbf{0.426} & 1.282  & 0.599 & 0.821 ($\pm$ 0.031) & 0.385 & 0.719    \\
    \hline
    \textbf{IS}    & 0.152 & 0.981  & \textbf{0.138} & 0.326 ($\pm$ 0.001) & 0.157 & 0.368    \\
    \bottomrule
    \end{tabular}
    \caption{Average solving time (in seconds) on the CO test sets.}
    \label{tab:learn-test-co-time}
\end{table}

Table \ref{tab:learn-test-co-time} quantifies the pivot efficiency across different benchmarks using varied rules. Notably, EXP-LEARN tends to solve in longer time, which is a byproduct of its intricate approach. Specifically, each pivot under EXP-LEARN necessitates a forward pass of the GCNN model. With a more meticulous design of the graph architecture, the time performance can be further optimized.

\section{Conclusion}\label{sec:conclusion}

The simplex methods are time-honored with rich practical design and mystery complexity. Generating a short path is the key task for pivot rules. In this paper, we focus on primal simplex and design two innovative and smart pivot experts that leverage both global and local information, i.e. optimal basis and steepest-edge score respectively. Experiments illustrate that these two experts overall outperform classical pivot rules significantly. To bridge theory to practical application, we integrate a GCNN model to learn knowledge from these experts. This imitation learning facilitates the circumvention of global information dependencies while preserving the performance in path generation. Empirical evidence confirms the learnability of our experts. Moreover, the learned rule commendably surpasses classical pivot rules in generating shorter pivot paths, although not quite caught up with the experts.

The value of our pivot experts extends beyond their standalone significance, serving both as benchmarks and generators of expert pivot labels. Particularly, the pivot experts outpace predecessors like MCTS in swiftly constructing superior paths, especially Expert I. Our learning paradigm effectively liberates us from the constraints of the simplex tableau. Furthermore, adopting a bipartite graph representation for LP proves more memory-conscious. Transitioning our method to dual simplex or primal-dual simplex, we anticipate, will be seamless with minimal adjustments. Yet, our reliance on GPUs poses challenges to solver implementation, potentially slowing down the simplex process through extended inference time. Generalizability remains a formidable hurdle within the ML4MO community. Addressing these challenges in future endeavors promises a tighter fusion between machine learning and optimization, fostering the advent of innovative ML4MO techniques.

\section*{Acknowledgments}

We thank Qi Huangfu for the fruitful discussions.

This research is partially supported by the National Natural Science Foundation of China (NSFC) [Grant NSFC-72150001, 72225009, 11831002].

\bibliographystyle{unsrt}
\bibliography{templateArxiv.bbl}

\newpage
\appendix
\section{Selected Features}\label{apx:feat}

Features used in our GCNN model are listed in Table \ref{tab:feat}.
\begin{table}[!ht]
    \centering
    \begin{tabular}{cll}
        \toprule
        Tensor                  & Feature                                   & Description                                 \\
        \midrule
        \multirow{4}{*}{Constr} & $\cos\langle \cbf, \abf \rangle$          & Cosine similarity with objective.           \\
                                & $\bbf$                                    & Right hand side.                            \\
                                & $\ybf$                                       & Dual solution value.                        \\
        \hline
        Edge                    & $\Abf_{i,j}$                               & Constraint coefficient.                     \\
        \hline
        \multirow{8}{*}{Var}    & $\cbf$                                    & Objective coefficient.                      \\
                                & has\_lb                                   & Lower bound indicator.                      \\
                                & has\_ub                                   & Upper bound indicator.                      \\
                                & lb                                   & Lower bound.                      \\
                                & ub                                   & Upper bound.                      \\
                                & basis\_status                             & One-hot basis status (basic, lower, upper). \\
                                & $\bar{\cbf}$                              & Reduced cost.                               \\
                                & $\xbf$                                    & Solution value.                             \\
                                & $\frac{\bar{\cbf}_{j}}{\sqrt{\|\Abf_{\Bbf}^{-1}\Abf_{j}\|^2+1}}$ & Score in the steepest-edge rule.               \\
        \bottomrule
    \end{tabular}
    \caption{Features used in GCNN.}\label{tab:feat}
\end{table}

\section{Implementation of the Simplex Solver}\label{apx:solver}

\subsection{Implementation Details}
In our simplex solver, we implement a user-friendly revised primal simplex method. We add tolerance for more robust performance. At the end of Phase I, we will detect and remove redundant constraints, so the LP instance solved in Phase II may be different from the original one. (This will not affect the fairness of our experiments because we provide the same modified LP for all pivot rules.)

We split the primal simplex method into many atomic operations for machine learning models to interact with. The solver allows one to pause and get information at each pivot, switch between existing pivot rules, or even customize his own pivot rule.

We do not implement a practical version of computing like sparse matrices or updating the inverse of the basic matrix, reduced cost, and score in each pivot rule. Therefore, our solver takes much longer time although it solves LP in a comparable pivot number to Gurobi.

\subsection{Capabilities and Limitations}
Our solver is a pure primal simplex prototype that allows one user to set a customized pivot rule. The solver can exactly solve over 80 out of 114 LP instances in NETLIB and all LP relaxation of mild-scale CO instances we generate. Since our matrices are stored in dense format for convenience, it is a pity that we cannot test too large LP with thousands of scales.

For purity of our simplex method, there is no advanced implementation in our solver such as scaling, Harris's ratio test, anti-degeneracy tricks, or other heuristics that can greatly improve the robustness of a simplex solver (see \cite{maros2012computational} for more details). Developing a high-performance simplex solver that supports the embedding or interaction of machine learning models has great significance but is far beyond the scope of this paper.

\section{Detailed Experiments}\label{apx:experiments}

\subsection{Reformulation}
LP instances from various applications share a more general form \eqref{prob:glp} containing both equality and inequality constraints
\begin{equation} \label{prob:glp}
    \begin{aligned}
        \min\  & \cbf^\top \xbf                                                            \\
        \st\   & \Abf_{\text{eq}}\xbf = \bbf_{\text{eq}}                                   \\
               & \lbf_{\text{constr}} \leq \Abf_{\text{ieq}}\xbf \leq \ubf_{\text{constr}} \\
               & \lbf_{\xbf} \leq \xbf \leq \ubf_{\xbf}.
    \end{aligned}
\end{equation}
To reformulate \eqref{prob:glp} as \eqref{prob:lp} for the primal simplex method, we add slacks $\sbf$ in those inequality constraints
\begin{equation}
    \begin{aligned}
        \min\  & \begin{bmatrix}
                     \cbf \\
                     \zerobf
                 \end{bmatrix}^\top \begin{bmatrix}
                                        \xbf \\
                                        \sbf
                                    \end{bmatrix}             \\
        \st\   & \begin{bmatrix}
                     \Abf_{\text{eq}}                 \\
                     -\Abf_{\text{ieq}} & \Ibf        \\
                     \Abf_{\text{ieq}}  &      & \Ibf
                 \end{bmatrix}
        \begin{bmatrix}
            \xbf \\
            \sbf
        \end{bmatrix}
        =
        \begin{bmatrix}
            \bbf_{\text{eq}}       \\
            - \lbf_{\text{constr}} \\
            \ubf_{\text{constr}}
        \end{bmatrix}                    \\
               & \begin{bmatrix}
                     \lbf_{\xbf} \\
                     \zerobf
                 \end{bmatrix}
        \leq
        \begin{bmatrix}
            \xbf \\
            \sbf
        \end{bmatrix}
        \leq
        \begin{bmatrix}
            \ubf_{\xbf} \\
            + \inftybm
        \end{bmatrix}.
    \end{aligned}
\end{equation}
We will remove those constraints and unnecessary slack components in $\sbf$ with an infinite right-hand side.

\subsection{Reproducing the Experiments}
The outcomes of reproducing the experiments could show small variations. Besides the randomness in experiments of NO-LOCAL and training EXP-LEARN, there exist slight differences in solving the linear system on different platforms, which will cause changes in the pivot path. This is a common phenomenon that occurs even with Gurobi's primal simplex method.

To overcome this issue and reasonably validate our approach, we have repeatedly conducted our experiments on different platforms and the results remain the same in general.

\subsection{Detailed Numerical Results}
Detailed numerical results on the NETLIB subset are presented in Table \ref{tab:apx-netlib-iter} and \ref{tab:apx-netlib-time}. Pivot numbers of COPT's dual simplex and Gurobi's primal simplex are presented for reference. The timeout may be caused by inefficient computation, extremely long paths, or cycling.

\paragraph{NETLIB subset}

\begin{center}\small
    \begin{longtable}[!ht]{lcc|cc|c|ccccc|cc}
    \toprule
        Prob	&	$m$	&	$n$	&	COPT	&	GUROBI	&	Phase I	&	Bland	&	Dantzig	&	SE	&	GI	&	LD	&	EXP	&	EXP-II\\
    \midrule
    \endfirsthead

    \toprule
        Prob	&	$m$	&	$n$	&	COPT	&	GUROBI	&	Phase I	&	Bland	&	Dantzig	&	SE	&	GI	&	LD	&	EXP	&	EXP-II\\
    \midrule
    \endhead

    \endfoot


p\_25fv47	&	696	&	1746	&	2794	&	1936	&	1428	&	t	&	7316	&	1136	&	1570	&	4649	&	984	&	834	\\
p\_adlittle	&	53	&	134	&	89	&	103	&	72	&	60	&	33	&	35	&	23	&	41	&	20	&	17	\\
p\_afiro	&	9	&	19	&	3	&	6	&	0	&	5	&	4	&	6	&	5	&	6	&	7	&	7	\\
p\_agg	&	149	&	234	&	82	&	116	&	164	&	77	&	31	&	27	&	27	&	31	&	25	&	24	\\
p\_agg2	&	277	&	471	&	143	&	169	&	0	&	175	&	130	&	138	&	124	&	146	&	124	&	121	\\
p\_agg3	&	277	&	471	&	142	&	170	&	0	&	225	&	144	&	129	&	144	&	143	&	123	&	113	\\
p\_bandm	&	178	&	333	&	257	&	286	&	276	&	1017	&	186	&	143	&	148	&	266	&	104	&	111	\\
p\_beaconfd	&	13	&	35	&	11	&	6	&	13	&	10	&	7	&	8	&	6	&	8	&	6	&	6	\\
p\_blend	&	54	&	89	&	59	&	50	&	64	&	87	&	50	&	39	&	37	&	39	&	31	&	18	\\
p\_bnl1	&	457	&	1336	&	957	&	1402	&	870	&	469	&	308	&	96	&	120	&	299	&	87	&	79	\\
p\_bnl2	&	992	&	2813	&	1133	&	2910	&	1593	&	14317	&	4142	&	974	&	984	&	3684	&	583	&	530	\\
p\_boeing1	&	290	&	655	&	359	&	500	&	475	&	1496	&	495	&	176	&	221	&	491	&	215	&	196	\\
p\_boeing2	&	122	&	261	&	113	&	172	&	167	&	185	&	86	&	43	&	50	&	68	&	37	&	34	\\
p\_bore3d	&	49	&	79	&	45	&	30	&	50	&	25	&	12	&	13	&	17	&	16	&	13	&	11	\\
p\_brandy	&	107	&	207	&	187	&	164	&	199	&	1128	&	131	&	112	&	115	&	191	&	73	&	67	\\
p\_capri	&	184	&	362	&	193	&	206	&	322	&	101	&	68	&	49	&	53	&	42	&	54	&	49	\\
p\_czprob	&	463	&	2446	&	635	&	2200	&	702	&	13060	&	838	&	449	&	620	&	1042	&	376	&	371	\\
p\_degen2	&	380	&	692	&	449	&	711	&	616	&	1974	&	718	&	263	&	287	&	2923	&	309	&	319	\\
p\_degen3	&	1410	&	2510	&	2018	&	4630	&	2517	&	t	&	t	&	923	&	1240	&	t	&	894	&	907	\\
p\_e226	&	143	&	364	&	252	&	273	&	193	&	2510	&	460	&	173	&	297	&	394	&	110	&	101	\\
p\_etamacro	&	316	&	609	&	411	&	545	&	548	&	748	&	253	&	193	&	182	&	331	&	160	&	137	\\
p\_fffff800	&	270	&	774	&	273	&	194	&	350	&	771	&	207	&	165	&	100	&	165	&	115	&	98	\\
p\_finnis	&	296	&	618	&	195	&	290	&	367	&	759	&	308	&	152	&	138	&	185	&	135	&	129	\\
p\_fit1d	&	24	&	1047	&	59	&	2270	&	0	&	t	&	1055	&	635	&	868	&	2686	&	1811	&	t	\\
p\_fit1p	&	627	&	1655	&	969	&	428	&	630	&	t	&	1027	&	467	&	603	&	483	&	394	&	380	\\
p\_fit2d	&	25	&	10387	&	33364	&	30970	&	0	&	t	&	13048	&	5754	&	t	&	60492	&	20376	&	t	\\
p\_forplan	&	93	&	398	&	184	&	179	&	232	&	491	&	103	&	62	&	66	&	62	&	36	&	36	\\
p\_ganges	&	382	&	721	&	402	&	554	&	384	&	603	&	431	&	357	&	397	&	463	&	488	&	433	\\
p\_gfrd-pnc	&	169	&	549	&	182	&	528	&	237	&	628	&	167	&	143	&	151	&	147	&	200	&	205	\\
p\_greenbeb	&	1415	&	3609	&	3182	&	3788	&	2889	&	t	&	5275	&	1623	&	t	&	2643	&	978	&	t	\\
p\_grow15	&	300	&	645	&	450	&	544	&	300	&	3537	&	105	&	105	&	79	&	112	&	147	&	169	\\
p\_grow22	&	440	&	946	&	723	&	886	&	440	&	3670	&	200	&	129	&	100	&	131	&	151	&	290	\\
p\_grow7	&	140	&	301	&	182	&	198	&	140	&	176	&	45	&	46	&	29	&	51	&	59	&	69	\\
p\_israel	&	163	&	304	&	113	&	160	&	172	&	407	&	297	&	110	&	170	&	110	&	78	&	82	\\
p\_kb2	&	37	&	62	&	28	&	41	&	37	&	151	&	48	&	47	&	87	&	51	&	17	&	16	\\
p\_ken-07	&	836	&	2031	&	1043	&	1438	&	1124	&	4411	&	985	&	905	&	984	&	956	&	1035	&	1036	\\
p\_lotfi	&	117	&	329	&	100	&	124	&	127	&	326	&	113	&	70	&	99	&	70	&	59	&	55	\\
p\_nesm	&	566	&	2733	&	1822	&	3338	&	1423	&	17815	&	3452	&	1294	&	2472	&	2254	&	2895	&	4263	\\
p\_osa-07	&	1047	&	24062	&	388	&	2716	&	1130	&	490	&	110	&	110	&	107	&	110	&	106	&	106	\\
p\_pds-02	&	1230	&	3900	&	837	&	557	&	1290	&	3194	&	947	&	660	&	638	&	936	&	534	&	543	\\
p\_perold	&	547	&	1321	&	1160	&	1519	&	2060	&	61786	&	5442	&	831	&	749	&	4068	&	671	&	573	\\
p\_pilot.ja	&	740	&	1726	&	1549	&	2034	&	2322	&	t	&	41609	&	1672	&	1607	&	16099	&	2895	&	3144	\\
p\_pilot4	&	352	&	946	&	554	&	682	&	631	&	33873	&	2328	&	456	&	709	&	1530	&	860	&	961	\\
p\_recipe	&	40	&	80	&	14	&	10	&	40	&	13	&	10	&	10	&	10	&	10	&	10	&	10	\\
p\_sc205	&	83	&	140	&	106	&	72	&	85	&	39	&	25	&	26	&	21	&	26	&	23	&	19	\\
p\_sc50b	&	18	&	30	&	21	&	13	&	19	&	1	&	1	&	1	&	1	&	1	&	1	&	1	\\
p\_scagr25	&	263	&	462	&	288	&	619	&	302	&	1267	&	306	&	196	&	248	&	296	&	271	&	206	\\
p\_scagr7	&	65	&	120	&	84	&	93	&	80	&	65	&	29	&	26	&	28	&	29	&	26	&	29	\\
p\_scfxm1	&	241	&	486	&	342	&	272	&	343	&	572	&	196	&	146	&	164	&	143	&	93	&	84	\\
p\_scfxm2	&	483	&	974	&	681	&	576	&	742	&	1715	&	313	&	275	&	999	&	352	&	186	&	157	\\
p\_scfxm3	&	725	&	1462	&	995	&	871	&	1066	&	t	&	491	&	343	&	532	&	474	&	254	&	226	\\
p\_scorpion	&	160	&	219	&	159	&	160	&	168	&	62	&	39	&	36	&	33	&	37	&	32	&	30	\\
p\_scrs8	&	177	&	919	&	328	&	333	&	254	&	1072	&	318	&	184	&	153	&	184	&	92	&	79	\\
p\_scsd1	&	77	&	760	&	107	&	217	&	86	&	817	&	197	&	264	&	199	&	264	&	267	&	t	\\
p\_scsd6	&	147	&	1350	&	324	&	404	&	163	&	t	&	398	&	256	&	319	&	256	&	129	&	117	\\
p\_sctap1	&	269	&	608	&	239	&	238	&	329	&	180	&	78	&	26	&	44	&	26	&	22	&	22	\\
p\_sctap2	&	977	&	2303	&	395	&	385	&	1118	&	1853	&	587	&	274	&	378	&	274	&	237	&	237	\\
p\_sctap3	&	1344	&	3111	&	518	&	418	&	1495	&	740	&	577	&	318	&	377	&	318	&	260	&	271	\\
p\_seba	&	2	&	9	&	1	&	8	&	8	&	11	&	6	&	6	&	6	&	6	&	7	&	7	\\
p\_share1b	&	93	&	228	&	117	&	138	&	143	&	433	&	192	&	105	&	59	&	89	&	55	&	44	\\
p\_share2b	&	93	&	159	&	110	&	90	&	143	&	158	&	48	&	27	&	24	&	26	&	25	&	24	\\
p\_shell	&	234	&	1157	&	273	&	397	&	259	&	408	&	127	&	129	&	137	&	128	&	111	&	110	\\
p\_ship04l	&	288	&	1901	&	403	&	423	&	308	&	303	&	76	&	71	&	92	&	71	&	70	&	70	\\
p\_ship04s	&	188	&	1253	&	295	&	269	&	192	&	254	&	75	&	72	&	70	&	72	&	65	&	64	\\
p\_ship08l	&	470	&	3121	&	656	&	789	&	474	&	3185	&	201	&	184	&	184	&	184	&	157	&	155	\\
p\_ship08s	&	234	&	1548	&	392	&	425	&	237	&	1943	&	115	&	94	&	105	&	96	&	86	&	85	\\
p\_ship12l	&	609	&	4170	&	877	&	1065	&	673	&	786	&	132	&	131	&	206	&	131	&	131	&	131	\\
p\_ship12s	&	267	&	1870	&	482	&	523	&	273	&	349	&	111	&	112	&	132	&	113	&	110	&	110	\\
p\_sierra	&	877	&	2294	&	415	&	352	&	986	&	2044	&	834	&	664	&	758	&	603	&	580	&	558	\\
p\_stair	&	258	&	407	&	446	&	291	&	340	&	725	&	237	&	126	&	191	&	418	&	92	&	82	\\
p\_standata	&	168	&	398	&	35	&	72	&	195	&	191	&	94	&	75	&	101	&	92	&	72	&	78	\\
p\_standgub	&	168	&	398	&	35	&	72	&	195	&	190	&	92	&	75	&	97	&	87	&	75	&	75	\\
p\_standmps	&	270	&	986	&	166	&	281	&	390	&	313	&	109	&	93	&	120	&	96	&	83	&	77	\\
p\_stocfor1	&	69	&	117	&	52	&	31	&	69	&	43	&	19	&	19	&	29	&	19	&	17	&	17	\\
p\_stocfor2	&	1448	&	2336	&	1166	&	915	&	1448	&	3153	&	639	&	604	&	1204	&	604	&	539	&	541	\\
p\_tuff	&	175	&	473	&	167	&	175	&	470	&	263	&	117	&	17	&	6	&	26	&	13	&	13	\\
p\_vtp.base	&	40	&	71	&	34	&	46	&	42	&	25	&	25	&	25	&	26	&	25	&	26	&	26	\\

    \bottomrule
    \caption{Pivot numbers on the NETLIB subset. "p\_" means the LP instance has been presolved. "t" means 300 seconds exceeded. COPT uses the dual simplex method while GUROBI uses the primal simplex method.}
    \label{tab:apx-netlib-iter} \\
\end{longtable}

\end{center}

\begin{center}
    \begin{longtable}[!ht]{lcc|ccccc|cc}
    \toprule
        Prob	&	$m$	&	$n$	&	Bland	&	Dantzig	&	SE	&	GI	&	LD	&	EXP	&	EXP-II \\
    \midrule
    \endfirsthead

    \toprule
        Prob	&	$m$	&	$n$	&	Bland	&	Dantzig	&	SE	&	GI	&	LD	&	EXP	&	EXP-II \\
    \midrule
    \endhead

    \endfoot


p\_25fv47	&	696	&	1746	&	t	&	21.786	&	6.014	&	90.052	&	15.575	&	3.639	&	16.853	\\
p\_adlittle	&	53	&	134	&	0.018	&	0.009	&	0.008	&	0.018	&	0.010	&	0.007	&	0.008	\\
p\_afiro	&	9	&	19	&	0.001	&	0.001	&	0.001	&	0.001	&	0.001	&	0.002	&	0.002	\\
p\_agg	&	149	&	234	&	0.046	&	0.020	&	0.020	&	0.030	&	0.021	&	0.020	&	0.025	\\
p\_agg2	&	277	&	471	&	0.176	&	0.097	&	0.117	&	0.293	&	0.119	&	0.115	&	0.265	\\
p\_agg3	&	277	&	471	&	0.160	&	0.105	&	0.109	&	0.317	&	0.118	&	0.112	&	0.221	\\
p\_bandm	&	178	&	333	&	0.549	&	0.095	&	0.085	&	0.294	&	0.142	&	0.068	&	0.158	\\
p\_beaconfd	&	13	&	35	&	0.002	&	0.001	&	0.002	&	0.002	&	0.002	&	0.002	&	0.002	\\
p\_blend	&	54	&	89	&	0.026	&	0.015	&	0.013	&	0.020	&	0.013	&	0.014	&	0.012	\\
p\_bnl1	&	457	&	1336	&	0.709	&	0.464	&	0.190	&	1.743	&	0.506	&	0.153	&	0.242	\\
p\_bnl2	&	992	&	2813	&	89.173	&	29.586	&	13.545	&	107.922	&	28.648	&	4.652	&	17.218	\\
p\_boeing1	&	290	&	655	&	1.206	&	0.387	&	0.177	&	0.805	&	0.406	&	0.205	&	0.522	\\
p\_boeing2	&	122	&	261	&	0.065	&	0.030	&	0.018	&	0.062	&	0.036	&	0.020	&	0.025	\\
p\_bore3d	&	49	&	79	&	0.005	&	0.003	&	0.003	&	0.007	&	0.004	&	0.005	&	0.005	\\
p\_brandy	&	107	&	207	&	0.340	&	0.040	&	0.040	&	0.146	&	0.065	&	0.031	&	0.053	\\
p\_capri	&	184	&	362	&	0.051	&	0.035	&	0.028	&	0.048	&	0.022	&	0.035	&	0.054	\\
p\_czprob	&	463	&	2446	&	21.855	&	1.446	&	1.461	&	25.602	&	2.108	&	0.795	&	3.640	\\
p\_degen2	&	380	&	692	&	2.210	&	0.835	&	0.378	&	1.739	&	3.504	&	0.443	&	1.168	\\
p\_degen3	&	1410	&	2510	&	t	&	t	&	18.909	&	170.144	&	t	&	10.910	&	53.267	\\
p\_e226	&	143	&	364	&	1.025	&	0.189	&	0.085	&	0.567	&	0.181	&	0.062	&	0.108	\\
p\_etamacro	&	316	&	609	&	0.631	&	0.217	&	0.199	&	0.561	&	0.301	&	0.171	&	0.360	\\
p\_fffff800	&	270	&	774	&	0.579	&	0.158	&	0.152	&	0.498	&	0.151	&	0.109	&	0.194	\\
p\_finnis	&	296	&	618	&	0.620	&	0.246	&	0.146	&	0.469	&	0.157	&	0.135	&	0.289	\\
p\_fit1d	&	24	&	1047	&	t	&	0.197	&	0.154	&	7.115	&	0.577	&	0.551	&	t	\\
p\_fit1p	&	627	&	1655	&	t	&	2.619	&	1.308	&	2.959	&	1.323	&	1.109	&	1.417	\\
p\_fit2d	&	25	&	10387	&	t	&	6.445	&	4.176	&	t	&	33.568	&	13.587	&	t	\\
p\_forplan	&	93	&	398	&	0.158	&	0.031	&	0.022	&	0.201	&	0.022	&	0.015	&	0.023	\\
p\_ganges	&	382	&	721	&	0.585	&	0.406	&	0.399	&	1.808	&	0.459	&	0.556	&	1.498	\\
p\_gfrd-pnc	&	169	&	549	&	0.339	&	0.071	&	0.072	&	0.278	&	0.070	&	0.117	&	0.344	\\
p\_greenbeb	&	1415	&	3609	&	t	&	69.536	&	49.962	&	t	&	47.861	&	18.032	&	t	\\
p\_grow15	&	300	&	645	&	2.892	&	0.081	&	0.099	&	0.322	&	0.093	&	0.141	&	0.493	\\
p\_grow22	&	440	&	946	&	4.634	&	0.251	&	0.198	&	0.852	&	0.174	&	0.220	&	1.655	\\
p\_grow7	&	140	&	301	&	0.103	&	0.021	&	0.021	&	0.035	&	0.022	&	0.030	&	0.075	\\
p\_israel	&	163	&	304	&	0.198	&	0.135	&	0.056	&	0.203	&	0.056	&	0.046	&	0.098	\\
p\_kb2	&	37	&	62	&	0.025	&	0.008	&	0.009	&	0.028	&	0.010	&	0.005	&	0.006	\\
p\_ken-07	&	836	&	2031	&	17.876	&	3.848	&	7.082	&	68.923	&	5.708	&	6.722	&	32.148	\\
p\_lotfi	&	117	&	329	&	0.113	&	0.045	&	0.029	&	0.146	&	0.029	&	0.027	&	0.049	\\
p\_nesm	&	566	&	2733	&	44.969	&	8.800	&	6.777	&	199.992	&	6.703	&	10.746	&	197.144	\\
p\_osa-07	&	1047	&	24062	&	22.419	&	5.107	&	5.363	&	9.777	&	5.379	&	5.089	&	6.025	\\
p\_pds-02	&	1230	&	3900	&	35.032	&	10.537	&	13.510	&	138.475	&	12.750	&	7.980	&	44.800	\\
p\_perold	&	547	&	1321	&	114.754	&	10.453	&	2.384	&	13.875	&	8.340	&	1.546	&	7.069	\\
p\_pilot.ja	&	740	&	1726	&	t	&	123.476	&	8.639	&	92.867	&	53.838	&	10.701	&	93.088	\\
p\_pilot4	&	352	&	946	&	25.051	&	2.248	&	0.614	&	7.780	&	1.621	&	1.044	&	8.731	\\
p\_recipe	&	40	&	80	&	0.003	&	0.002	&	0.002	&	0.003	&	0.002	&	0.003	&	0.003	\\
p\_sc205	&	83	&	140	&	0.015	&	0.010	&	0.011	&	0.015	&	0.009	&	0.009	&	0.011	\\
p\_sc50b	&	18	&	30	&	0.000	&	0.000	&	0.000	&	0.000	&	0.000	&	0.001	&	0.000	\\
p\_scagr25	&	263	&	462	&	0.911	&	0.211	&	0.163	&	0.501	&	0.215	&	0.234	&	0.381	\\
p\_scagr7	&	65	&	120	&	0.022	&	0.010	&	0.010	&	0.014	&	0.007	&	0.009	&	0.018	\\
p\_scfxm1	&	241	&	486	&	0.377	&	0.123	&	0.108	&	0.331	&	0.103	&	0.075	&	0.135	\\
p\_scfxm2	&	483	&	974	&	2.438	&	0.448	&	0.507	&	3.199	&	0.546	&	0.331	&	0.840	\\
p\_scfxm3	&	725	&	1462	&	t	&	1.452	&	1.311	&	7.905	&	1.519	&	0.872	&	2.974	\\
p\_scorpion	&	160	&	219	&	0.040	&	0.020	&	0.017	&	0.024	&	0.018	&	0.018	&	0.022	\\
p\_scrs8	&	177	&	919	&	0.571	&	0.173	&	0.124	&	0.934	&	0.124	&	0.063	&	0.097	\\
p\_scsd1	&	77	&	760	&	0.236	&	0.058	&	0.096	&	0.536	&	0.093	&	0.107	&	t	\\
p\_scsd6	&	147	&	1350	&	t	&	0.286	&	0.167	&	1.391	&	0.170	&	0.085	&	0.159	\\
p\_sctap1	&	269	&	608	&	0.181	&	0.057	&	0.022	&	0.101	&	0.022	&	0.020	&	0.025	\\
p\_sctap2	&	977	&	2303	&	9.608	&	3.069	&	1.889	&	13.572	&	1.879	&	1.466	&	4.236	\\
p\_sctap3	&	1344	&	3111	&	7.808	&	6.147	&	4.558	&	27.905	&	4.541	&	3.150	&	9.368	\\
p\_seba	&	2	&	9	&	0.002	&	0.001	&	0.001	&	0.001	&	0.001	&	0.001	&	0.002	\\
p\_share1b	&	93	&	228	&	0.135	&	0.055	&	0.035	&	0.060	&	0.029	&	0.023	&	0.031	\\
p\_share2b	&	93	&	159	&	0.065	&	0.018	&	0.009	&	0.015	&	0.008	&	0.010	&	0.013	\\
p\_shell	&	234	&	1157	&	0.303	&	0.087	&	0.109	&	0.926	&	0.095	&	0.093	&	0.188	\\
p\_ship04l	&	288	&	1901	&	0.291	&	0.076	&	0.078	&	0.407	&	0.078	&	0.082	&	0.170	\\
p\_ship04s	&	188	&	1253	&	0.147	&	0.041	&	0.046	&	0.155	&	0.045	&	0.044	&	0.068	\\
p\_ship08l	&	470	&	3121	&	6.559	&	0.414	&	0.591	&	2.721	&	0.538	&	0.398	&	0.932	\\
p\_ship08s	&	234	&	1548	&	1.453	&	0.082	&	0.077	&	0.392	&	0.079	&	0.073	&	0.123	\\
p\_ship12l	&	609	&	4170	&	3.514	&	0.611	&	0.718	&	3.590	&	0.782	&	0.690	&	1.149	\\
p\_ship12s	&	267	&	1870	&	0.305	&	0.073	&	0.089	&	0.371	&	0.090	&	0.092	&	0.185	\\
p\_sierra	&	877	&	2294	&	9.001	&	3.705	&	5.369	&	49.276	&	3.167	&	3.701	&	20.187	\\
p\_stair	&	258	&	407	&	0.488	&	0.164	&	0.101	&	0.531	&	0.298	&	0.078	&	0.164	\\
p\_standata	&	168	&	398	&	0.104	&	0.045	&	0.041	&	0.124	&	0.045	&	0.045	&	0.089	\\
p\_standgub	&	168	&	398	&	0.104	&	0.044	&	0.041	&	0.118	&	0.043	&	0.046	&	0.097	\\
p\_standmps	&	270	&	986	&	0.248	&	0.088	&	0.093	&	0.514	&	0.082	&	0.083	&	0.150	\\
p\_stocfor1	&	69	&	117	&	0.010	&	0.005	&	0.005	&	0.015	&	0.005	&	0.006	&	0.007	\\
p\_stocfor2	&	1448	&	2336	&	29.449	&	6.350	&	7.660	&	113.607	&	7.475	&	6.221	&	21.390	\\
p\_tuff	&	175	&	473	&	0.133	&	0.058	&	0.010	&	0.009	&	0.015	&	0.009	&	0.011	\\
p\_vtp.base	&	40	&	71	&	0.004	&	0.004	&	0.004	&	0.011	&	0.003	&	0.006	&	0.008	\\

        \bottomrule
        \caption{Solving time (in seconds) on the NETLIB subset. "p\_" means the LP instance has been presolved. "t" means 300 seconds exceeded.}
    \label{tab:apx-netlib-time} \\
    \end{longtable}
\end{center}

\end{document}